\documentclass[11pt]{article}
\usepackage{amsmath}
\usepackage{amsfonts}
\usepackage{amssymb}
\usepackage{enumerate}
\usepackage{color}
\usepackage{lscape}
\usepackage{latexsym}
\usepackage{multirow}
\usepackage{graphicx}
\newtheorem{theorem}{\bf Theorem}[section]
\newtheorem{lemma}[theorem]{\bf Lemma}
\newtheorem{prop}[theorem]{\bf Proposition}

\newtheorem{coro}[theorem]{\bf Corollary}

\newenvironment{proof}{\noindent{\em Proof:}}{\quad \hfill$\Box$\vspace{2ex}}

\def \bN {\Bbb N}
\def \bZ {\Bbb Z}

\def \bR {\Bbb R}

\def \bC {\Bbb C}

\def \bx {{\bf x}}
\def \by {{\bf y}}

\def \bz {{\bf z}}

\def \cB {{\cal B}}

\def \cH {{\cal H}}

\def \cE {{\cal E}}

\def \cW {{\cal W}}

\def \and {\, \mbox{\rm and}\, }
\def \sinc {\,{\rm sinc}\,}

\def \span {\,{\rm span}\,}
\def \supp {\,{\rm supp}\,}

\makeatletter

\newcommand{\Rmnum}[1]{\expandafter\@slowromancap\romannumeral #1@}
\makeatother
\begin{document}

\title{\bf On the Inclusion Relation of Reproducing Kernel Hilbert Spaces}
\author{Haizhang Zhang\thanks{School of Mathematics
and Computational Science and Guangdong Province Key Laboratory of Computational Science,
 Sun Yat-sen University, Guangzhou 510275, P. R. China. E-mail address: {\it zhhaizh2@sysu.edu.cn}. Supported in part by Guangdong Provincial Government of China through the
``Computational Science Innovative Research Team" program.}\quad and \quad Liang Zhao
\thanks{Department of Mathematics, Syracuse University, Syracuse, NY 13244, USA. E-mail address: {\it lzhao04@syr.edu}. Supported in part by US Air Force Office of Scientific Research under grant FA9550-09-1-0511.}}
\date{}
\maketitle

\begin{abstract}To help understand various reproducing kernels used in applied sciences, we investigate the inclusion relation of two reproducing kernel Hilbert spaces. Characterizations in terms of feature maps of the corresponding reproducing kernels are established. A full table of inclusion relations among widely-used translation invariant kernels is given. Concrete examples for Hilbert-Schmidt kernels are presented as well. We also discuss the preservation of such a relation under various operations of reproducing kernels.
Finally, we briefly discuss the special inclusion with a norm equivalence.
\end{abstract}


\vspace*{0.3cm}

\section{Introduction}
\setcounter{equation}{0}

Reproducing kernel Hilbert spaces (RKHS) are Hilbert spaces of functions on which point evaluations are always continuous linear functionals. They are the natural choice of background spaces for many applications. First of all, thanks to the existence of an inner product, Hilbert spaces are the normed vector spaces that are well-understood and can be handled best. Secondly, the inputs for many application-oriented algorithms are usually modeled as the sample data of some desirable but unknown function. Requiring the sampling process to be stable seems to be a necessity. Mathematically, this is synonymous with desiring point evaluation functionals to be bounded. For these reasons, RKHS are widely applicable in probability and statistics \cite{Berlinet,Sriperumbudur2010}, dimension reduction \cite{Fukumizu2004}, numerical study of differential equations \cite{Cui2007,Franke1998}, generalizations of the Shannon sampling theory \cite{NW,Zhanga}, and approximation from scattered data \cite{Wendland}. Moreover, an RKHS possesses a unique function, named a reproducing kernel, which represents point evaluations on the space. Reproducing kernels are able to measure the similarity between inputs and could save the calculation of inner products in a feature space \cite{SS}. This gives birth to the ``kernel trick" in machine learning and makes RKHS the popular underlying feature spaces for applications in the field. As a result, reproducing kernel based methods are dominant in machine learning \cite{CuckerSmale,Evgeniou2000,SS,SC,Vapnik1998}.

Despite the wide applications of RKHS, there are some important theoretical issues that are not well-understood. This paper is devoted to the inclusion relation between RKHS, that is, given two reproducing kernels, we are interested in whether the RKHS of one reproducing kernel is contained by the RKHS of the other. The clarification of this problem is helpful to understand the structure of RKHS and hence is contributive to the theory of reproducing kernels \cite{Aronszajn}. For instance, the relation is needed in building a multi-resolution decomposition of RKHS. Besides, the study could provide guidelines to the choice of reproducing kernels in machine learning. There are many reproducing kernels in the literature. In a particular application, the selection of reproducing kernels is usually critical to the success of a learning algorithm. While there are no well-recognized guidelines in making such a decision, avoiding overfitting or underfitting is usually the first principle. When overfitting or underfitting occurs, a remedy is to change the current reproducing kernel so that the RKHS of the new kernel becomes smaller or larger compared to that of the existing kernel. Understanding the inclusion relation between RKHS could help achieve such an update of reproducing kernels.

Three characterizations of the inclusion relation of RKHS were established before 1970s \cite{Aronszajn,DR,Ylvisaker}. With the advent of machine learning in 1990s, there has been increasing interest in reproducing kernels and RKHS. Many concrete reproducing kernels have emerged in the literature and in applications. Most of them can be conveniently represented by a feature map, which was unknown in the past studies \cite{Aronszajn,DR,Ylvisaker}. The purpose of this paper is to provide a systematic study of the inclusion relation of RKHS with focus on the concrete examples of RKHS appeared in machine learning. Recent references \cite{XZ,XZ2} studied the embedding relation of RKHS, that is, an equal norm requirement is imposed. As shown by the examples therein, the requirement that two RKHS share the same norm on the smaller space might be demanding and rules out many commonly-used RKHS. For example, the RKHS of a Gaussian kernel can not be properly embedded into the RKHS of another translation invariant reproducing kernel of a continuous type. By relaxing the requirement, we shall see more applications and have more structural results.

The outline of the paper is as follows. We shall discuss characterizations of the inclusion relation in the next section. Sections 3 and 4 are devoted to the investigation of concrete translation invariant and Hilbert-Schmidt reproducing kernels, respectively. Particularly, we shall establish a full table of inclusion relations among popular translation invariant reproducing kernels in Section 3. In Section 5, we discuss the preservation of the relation under various operations of reproducing kernels. In the last section, we shall briefly discuss the special inclusion relation where a norm equivalence is required.

\section{Characterizations}
\setcounter{equation}{0}

We start with introducing some basics of the theory of reproducing kernels \cite{Aronszajn}. Let $X$ be a prescribed set, which is often referred to as an input space in machine learning. A {\it reproducing kernel} (or kernel for short) $K$ on $X$ is a function from $X\times X$ to $\bC$ such that for all finite pairwise distinct inputs $\bx:=\{x_j:j\in\bN_n\}\subseteq X$, the kernel matrix
$$
K[\bx]:=\left[K(x_j,x_k):j,k\in\bN_n\right]
$$
is hermitian and positive semi-definite. Here, for the simplicity of enumerating with finite sets, we denote for each $n\in\bN$ by $\bN_n:=\{1,2,\ldots,n\}$. A reproducing kernel $K$ on $X$ corresponds to a unique RKHS, denoted by $\cH_K$, such that $K(x,\cdot)\in\cH_K$ for all $x\in X$ and
\begin{equation}\label{reproducing}
f(x)=(f,K(x,\cdot))_{\cH_K}\quad \mbox{ for all }f\in\cH_K,\ x\in X,
\end{equation}
where $(\cdot,\cdot)_{\cH_K}$ denotes the inner product on $\cH_K$. There is a characterization of reproducing kernels in terms of feature maps. A {\it feature map} for a kernel $K$ on $X$ is a mapping from $X$ to another Hilbert space $\cW$ such that
\begin{equation}\label{featuremap}
K(x,y)=(\Phi(x),\Phi(y))_\cW,\ \ x,y\in X.
\end{equation}
The space $\cW$ is call a {\it feature space} for kernel $K$. One observes from (\ref{reproducing}) that
$$
K(x,y)=(K(x,\cdot),K(y,\cdot))_{\cH_K},\ \ x,y\in X.
$$
Thus, $\Phi(x):=K(x,\cdot)$, $x\in X$ and $\cW:=\cH_K$ is a pair of feature map and feature space for $K$. The RKHS of a reproducing kernel can be easily identified once a feature map representation is available. The following result is well-known in machine learning community \cite{Opfer,SS,XZ}.

For a feature map $\Phi:X\to\cW$, we shall denote by $P_\phi$ the orthogonal projection from $\cW$ onto the linear span $\span\Phi(X)$ of $\Phi(X)$.

\begin{lemma}\label{featuremapconstruction}
If $K$ is a kernel on $X$ given by (\ref{featuremap}) by a feature map $\Phi$ from $X$ to $\cW$ then
$\cH_K = \{(\Phi(\cdot), u)_{\cW} : u \in \cW\}$ with the inner product
 \begin{equation}\label{innerproductbyfeaturemap}
 (u,\Phi(\cdot))_{\cW}, (v,\Phi(\cdot))_{\cW})_{\cH_K} = (P_{\Phi}u, P_{\Phi}v)_{\cW}, \quad u, v \in \cW.
 \end{equation}
In particular, if $\span\Phi(X)$ is dense in $\cW$ then $\cH_K$ is isometrically isomorphic to $\cW$ through the linear mapping $(u,\Phi(\cdot))_\cW\to u$.
\end{lemma}

As an example, we look at the sinc kernel
$$
K(x,y)=\sinc(x-y):=\prod_{j=1}^d \frac{\sin\pi(x_j-y_j)}{\pi(x_j-y_j)},\ \ x,y\in\bR^d.
$$
It can be represented as the Fourier transform of $\frac1{\sqrt{2\pi}^d}\chi_{[-\pi,\pi]^d}$, where $\chi_A$ is the characteristic function of a subset $A\subseteq\bR^d$. In this paper, we adopt the following forms of the Fourier transform $\hat{f}$ and the inverse Fourier transform $\check{f}$ of a Lebesgue integrable function $f\in L^1(\bR^d)$
$$
\hat{f}(\xi):=\biggl(\frac1{\sqrt{2\pi}}\biggr)^d\int_{\bR^d}f(x)e^{-i(x, \xi)}dx,\ \ \check{f}(\xi):=\biggl(\frac1{\sqrt{2\pi}}\biggr)^d\int_{\bR^d}f(x)e^{i(x,\xi)}dx,\ \ \xi\in\bR^d.
$$
Here $(x, \xi)$ is the standard inner product on $\bR^d$. Thus, one sees that
\begin{equation}\label{sinc}
\sinc(x-y)=\frac1{(2\pi)^d}\int_{[-\pi,\pi]^d}e^{-i(\xi,x-y)}d\xi,\ \ x,y\in\bR^d.
\end{equation}
Thus $\cW:=L^2([-\pi,\pi]^d)$ and $\Phi(x):=(\frac1{\sqrt{2\pi}})^de^{-i(\xi, x)}$, $x\in\bR^d$ satisfy (\ref{featuremap}). Lemma \ref{featuremapconstruction} tells that $\cH_K$ is the space of continuous square integrable functions on $\bR^d$ whose Fourier transforms are supported on $[-\pi,\pi]^d$ and the inner product on $\cH_K$ inherits from that of $L^2(\bR^d)$. This is well-known. We use it to illustrate the application of Lemma \ref{featuremapconstruction}.

Given two kernels $K,G$ on a prescribed input space $X$, the corresponding RKHS $\cH_K,\cH_G$ can usually be identified
by Lemma \ref{featuremapconstruction}. The theme of the paper is the set inclusion relation
$\cH_K\subseteq \cH_G$. As point evaluations are continuous on RKHS,
it was observed in \cite{Aronszajn} that if $\cH_K\subseteq\cH_G$ then the identity operator
from $\cH_K$ into $\cH_G$ is bounded. We shall denote by $\beta(K,G)$ the operator norm of this embedding.
A characterization of $\cH_K\subseteq\cH_G$ was also established in \cite{Aronszajn}. Following \cite{Aronszajn}, we write $K \ll G$ if $G - K$ remains a kernel on $X$.

\begin{lemma}\label{Aronszajn}\cite{Aronszajn}
Let $K, G$ be two kernels on X, $\cH_K \subseteq \cH_G$ if and only if there exists a
nonnegative constant $\lambda \ge 0$ such that $K \ll \lambda  G$.
\end{lemma}

Provided that $\cH_K\subseteq\cH_G$, we shall denote by $\lambda(K,G)$ the infimum of the set of positive constants $\lambda$
such that $K\ll \lambda G$. If $\cH_K\nsubseteq\cH_G$ then we make the convention that $\lambda(K,G)=+\infty$.
We first make a simple observation about the two quantities $\beta(K,G)$ and $\lambda(K,G)$.

\begin{prop}\label{twoquantities}
Let $K,G$ be two kernels on $X$ with $\cH_K\subseteq\cH_G$ then $\beta(K,G)=\sqrt{\lambda(K,G)}$ and $K\ll \lambda(K,G) G$.
\end{prop}
\begin{proof}
It was proved in \cite{Aronszajn} that for two kernels $K,L$ on $X$, $K\ll L$ if and only if $\cH_K\subseteq \cH_L$ and $\|f\|_{\cH_L}\le \|f\|_{\cH_K}$ for all $f\in\cH_K$.
Note by Lemma \ref{featuremapconstruction} that $\cH_G$ and $\cH_{\lambda G}$ share common elements and for all $f\in\cH_G$ that
$$
\|f\|_{\cH_G}=\sqrt{\lambda} \|f\|_{\cH_{\lambda G}}.
$$
Combing these two facts, we get that for all $\lambda>0$ that $K\ll \lambda G$ if and only if $\cH_K\subseteq\cH_G$ and
$\|f\|_{\cH_G}\le \sqrt{\lambda} \|f\|_{\cH_K}$ for all $f\in\cH_K$. Thus, if $K\ll \lambda G$ then
$\beta(K,G)\le \sqrt{\lambda}$. It follows that $\beta(K,G)\le \sqrt{\lambda(K,G)}$. On the other hand, if $\beta>\beta(K,G)$
then there exists some $f\in\cH_K$ for which either $f\notin \cH_G$ or $\|f\|_{\cH_G}>\beta \|f\|_{\cH_K}$.
It implies that $K\ll \beta^2G$ does not hold. As a consequence, $\lambda(K,G)\le \beta^2$. We hence have that $\sqrt{\lambda(K,G)}\le \beta(K,G)$, leading to the equality
$$
\beta(K,G)=\sqrt{\lambda(K,G)},
$$ which in turn implies that $K\ll \lambda(K,G) G$.
\end{proof}

We next present another characterization of the inclusion relation in terms of feature maps of reproducing kernels.
\begin{theorem}\label{featuremapcharacterization}
Let $K,G$ be two kernels on $X$ with the feature map $\Phi_1 : X \to \cW_1$ and $\Phi_2:X\to\cW_2$, respectively. If $\overline{\span} \Phi_1(X) = \cW_1$ and $\overline{\span} \Phi_2(X) = \cW_2$ then $\cH_K \subseteq \cH_G$ if and only if there exists a bounded linear operator $T: \cW_2 \rightarrow \cW_1$
such that
\begin{equation}\label{operatorT}
T\Phi_2(x) = \Phi_1(X), \quad x \in X.
\end{equation}
Moreover, the inclusion is nontrivial if and only if the adjoint operator $T^*$ of $T$ is not surjective.
\end{theorem}
\begin{proof}
The result can be proved by similar arguments as those in Theorems 6 and 7 of \cite{XZ}.
\end{proof}

By the above theorem, the particular choices
$$
\cW_1:=\cH_K,\ \Phi_1(x):=K(x,\cdot),\quad \cW_2:=\cH_G,\ \Phi_2(x)=G(x,\cdot),\ x\in X
$$
yields that $\cH_K\subseteq\cH_G$ if and only if there exists a bounded operator $L: \cH_G \rightarrow \cH_K$ such that $LG(x, \cdot) = K(x, \cdot)$ for all $x \in X$. We remark that this result in the special case when $X$ is a countable dense subset of $\bR^d$ was proved in \cite{DR}.

\section{Translation Invariant Kernels and Radial Basis Functions}
\setcounter{equation}{0}

Translation invariant kernels are the most widely-used class of reproducing kernels on the Euclidean space. A kernel $K$ on $\bR^d$ is said to be {\it translation invariant} if
$$
K(x - a, y-a) = K(x, y) \mbox{ for all }x,y,a\in\bR^d.
$$
There is a celebrated characterization of continuous translation invariant kernels on $\bR^d$ due to Bochner \cite{Bochner}. The result is usually referred to as the Bochner theorem. Denote by $\cB(\bR^d)$ the set of all the finite positive Borel measures on $\bR^d$. The characterization states that continuous translation invariant kernels on $\bR^d$ are exactly the Fourier transform of finite positive Borel measures in $\cB(\bR^d)$. Thus we shall consider the inclusion relation $\cH_K\subseteq\cH_G$ for two translation invariant kernels $K,G$ of the form
\begin{equation}\label{translationK}
K(x,y) = \int_{\bR^d} e^{i(x-y, \xi)} \,d\mu(\xi), \quad x, y \in \bR^d,
\end{equation}
and
\begin{equation}\label{translationG}
G(x,y) = \int_{\bR^d} e^{i(x-y, \xi)} \,d\nu(\xi), \quad x, y \in \bR^d ,
\end{equation}
where $\mu,\nu\in\cB(\bR^d)$.

Let $\mu,\nu$ be two finite Borel measures on a topological space $Y$. Recall that $\mu$ is said to be {\it absolutely continuous}
 with respect to $\nu$, denoted as $\mu\ll\nu$, if $\mu$ vanishes on Borel subsets of $Y$ with zero $\nu$ measure.
 When $\mu\ll\nu$, $d\mu/d\nu$ is a Borel measurable function on $Y$ such that
$$
\mu(A)=\int_A \frac{d\mu}{d\nu}(x)d\nu(x)\ \ \mbox{ for all Borel subsets }A\subseteq Y.
$$
We denote by $L^\infty_\nu(Y)$ the space of Borel measurable functions on $Y$ with the norm
$$
\|f\|_{L^\infty_\nu(Y)}:=\inf\{M>0:\nu(\{t\in Y:|f(t)|>M\})=0\}<+\infty.
$$
For later use, we also denote by $L^2_\nu(Y)$ the Hilbert space of Borel measurable functions on $Y$ such that
$$
\|f\|_{L^2_\nu(Y)}:=\left(\int_Y |f(t)|^2d\nu(t)\right)^{1/2}<+\infty.
$$

\begin{prop}\label{characterizationtranslation}
Let $K,G$ be two continuous translation invariant kernels on $\bR^d$ given by (\ref{translationK}) and (\ref{translationG}). Then
 $\cH_K\subseteq \cH_G$ if and only if $\mu\ll\nu$ and $d\mu/d\nu\in L^\infty_\nu(\bR^d)$. In the case that $\cH_K\subseteq\cH_G$,
\begin{equation}\label{lambdatranslation}
\lambda(K,G)=\left\|\frac{d\mu}{d\nu}\right\|_{L^\infty_\nu(\bR^d)}.
\end{equation}
\end{prop}
\begin{proof}
By Lemma \ref{Aronszajn}, $\cH_K\subseteq\cH_G$ if and only if there exists some $\lambda\ge0$ such that $\lambda G-K$ is a kernel on $\bR^d$. Note that for all $\lambda\ge0$, $\lambda G-K$ is still translation invariant. Therefore, by the Bochner theorem, $K\le \lambda G$ if and only if $\lambda\nu-\mu\in\cB(\bR^d)$, which happens if and only if $\mu\ll\nu$ and $d\mu/d\nu$ is bounded by $\lambda$ almost everywhere on $\bR^d$ with respect to $\nu$. We hence get that $\cH_K\subseteq \cH_G$ if and only if $\mu\ll\nu$ and $d\mu/d\nu\in L^\infty_\nu(\bR^d)$. When $\mu\ll\nu$ and $d\mu/d\nu\in L^\infty_\nu(\bR^d)$, it is clear that (\ref{lambdatranslation}) holds.
\end{proof}

We pay special attention to the situation when the Borel measures in (\ref{translationK}) and (\ref{translationG}) are absolutely continuous with respect to the Lebesgue measure. In this case, by the Radon-Nikodym theorem, $K,G$ are the Fourier transform of nonnegative Lebesgue integrable functions on $\bR^d$.

\begin{coro}\label{chatranslationcontinuous}
Let $u,v$ be nonnegative functions in $L^1(\bR^d)$ and let $K,G$ be defined by
\begin{equation}\label{translationKG2}
K(x,y)=\int_{\bR^d}e^{i(x-y,\xi)}u(\xi)d\xi,\ \ G(x,y)=\int_{\bR^d}e^{i(x-y,\xi)}v(\xi)d\xi,\ \ x,y\in\bR^d.
\end{equation}
Then $\cH_K\subseteq\cH_G$ if and only if the set $\{t\in\bR^d:u(t)>0,\ v(t)=0\}$ has Lebesgue measure zero and
$u/v$ is essentially bounded on $\{t\in\bR^d:v(t)>0\}$, in which case $\lambda(K,G)$ equals the essential upper bound of $u/v$
on $\{t\in\bR^d:v(t)>0\}$. In particular, if $v$ is positive almost everywhere on $\bR^d$ then
$\cH_K\subseteq\cH_G$ if and only if $u/v\in L^\infty(\bR^d)$, in which case $\lambda(K,G)=\|u/v\|_{L^\infty(\bR^d)}$.
\end{coro}

An important class of translation invariant kernels on $\bR^d$ are given by radial basis functions. Those are reproducing kernels of the form
\begin{equation}\label{RBFg}
K_d(x,y)=g(\|x-y\|), \ \ x,y\in\bR^d,
\end{equation}
where $g$ is a single-variate function on $\bR_+:=[0,+\infty)$ and $\|\cdot\|$ is the standard Euclidean norm on $\bR^d$. The following well-known
characterizations of kernels of the form (\ref{RBFg}) are due to Schoenberg \cite{SIJ}.

For each $d\in\bN$, denote by $d\omega_d$ and $\omega_d$ the area element and total area of the unit sphere of $\bR^d$, respectively.
Also set
$$
\Omega_d(|x|):=\frac{1}{\omega_d}\int_{\|\xi\|=1}e^{i(x,\xi)}d\omega_d(\xi),\ \ x\in\bR^d.
$$

\begin{lemma}\label{Schoenberg}
Let $g$ be a function on $\bR_+$. Then (\ref{RBFg}) defines a reproducing kernel on $\bR^d$ if and only if there is a finite positive Borel
measure $\mu$ on $\bR_+$ such that
\begin{equation}\label{RBFdkernelK}
K_{d}(x,y)=\int_{0}^\infty\Omega_d(t\|x-y\|)d\mu(t),\ \ x,y\in\bR^d.
\end{equation}
Furthermore, equation (\ref{RBFg}) defines a reproducing kernel $K_d$ on $\bR^d$ for all $d\in\bN$ if and only if
\begin{equation}\label{RBFallkernelK}
K_{d}(x,y)=\int_{0}^\infty e^{-t\|x-y\|^2}d\mu(t),\ \ x,y\in\bR^d
\end{equation}
for some finite positive Borel measure $\mu$ on $\bR_+$.
\end{lemma}

Notice that both $\span\{\Omega_d(tr):r>0\}$ and $\span\{e^{-tr}:r>0\}$ are dense in $C_0(\bR_+)$, the space of continuous functions
on $\bR_+$ vanishing at infinity equipped with the maximum norm. By this fact and Lemma \ref{Schoenberg}, one may use arguments similar to those in
the proof of Proposition \ref{characterizationtranslation} to get the following characterizations of the inclusion relation of RKHS of
kernels of the form (\ref{RBFg}).

\begin{prop}\label{characterizeRBF}
Let $\mu,\nu$ be two finite positive Borel measures on $\bR_+$, let $K_d$ be given by (\ref{RBFdkernelK}) and set
\begin{equation}\label{RBFdkernelG}
G_{d}(x,y):=\int_{0}^\infty\Omega_d(t\|x-y\|)d\nu(t),\ \ x,y\in\bR^d.
\end{equation}
Then $\cH_{K_d}\subseteq\cH_{G_d}$ if and only if $\mu\ll\nu$ and $d\mu/d\nu\in L^\infty_\nu(\bR_+)$, in which case
$\lambda(K_d,G_d)=\|d\mu/d\nu\|_{L^\infty_\nu(\bR_+)}$.

If $K_d$ is given by (\ref{RBFallkernelK}) and $G_d$ is defined by
\begin{equation}\label{RBFallkernelG}
G_d(x,y)=\int_{0}^\infty e^{-t\|x-y\|^2}d\nu(t),\ \ x,y\in\bR^d
\end{equation}
then $\cH_{K_d}\subseteq\cH_{G_d}$ for all $d\in\bN$ and $\{\lambda(K_d,G_d):d\in\bN\}$ is bounded if and only if $\mu\ll\nu$ and
$d\mu/d\nu\in L^\infty_\nu(\bR_+)$, in which case $\sup\{\lambda(K_d,G_d):d\in\bN\}=\|d\mu/d\nu\|_{L^\infty_\nu(\bR_+)}$.
\end{prop}

One may specify statements in the above proposition to the case when $\mu,\nu$ are absolutely continuous with respect to the Lebesgue
measure on $\bR_+$ to get results similar to those in Corollary \ref{chatranslationcontinuous}, which we shall not state here.

We next turn to the main purpose of this section, which is to explore the inclusion relations among the RKHS of six commonly used
translation invariant kernels in machine learning and other areas of applied mathematics. To apply the characterizations established
above, we present those kernels in the form they appear in the characterization of Bochner or Schoenberg:
\begin{itemize}

\item[--] the Gaussian kernel
\begin{equation}\label{gaussian}
G_\gamma(x,y)=\exp\biggl(-\frac{\|x-y\|^2}{\gamma}\biggr)=
\int_{\bR^d}e^{i(x-y,\xi)}g_\gamma(\xi)d\xi,\ \ x,y\in\bR^d,\ \gamma>0.
\end{equation}
where
$$
g_\gamma(\xi):=\biggl(\frac{\sqrt{\gamma}}{2\sqrt{\pi}}\biggr)^d\exp(-\frac{\gamma\|\xi\|^2}4),\ \ \xi\in\bR^d.
$$

\item[--] the $\ell^1$-norm exponential kernel
\begin{equation}\label{l1exponential}
E_{\sigma_1}(x,y)=\exp\biggl(-\frac{\|x-y\|_1}{\sigma_1}\biggr)=\int_{\bR^d}e^{i(x-y,\xi)}\varphi_{\sigma_1}(\xi)d\xi,\ \ x,y\in\bR^d,\ \sigma_1>0,
\end{equation}
where $\|x\|_1:=\sum_{j=1}^d|x_j|$, $x=(x_j:j\in\bN_d)\in\bR^d$ and
$$
\varphi_{\sigma_1}(\xi):=\frac{\sigma_1^d}{\pi^d}\prod_{j=1}^d\frac1{1+\sigma_1^2\xi_j^2},\ \ \xi\in\bR^d.
$$

\item[--] the $\ell^2$-norm exponential kernel
\begin{equation}\label{l2exponential}
\cE_{\sigma_2}(x,y)=\exp(-\|x-y\|)=\int_{\bR^d}e^{i(x-y,\xi)}\psi_{\sigma_2}(\xi)d\xi,
\ \ x,y\in\bR^d,\ \sigma_2>0,
\end{equation}
where
\begin{equation}\label{l2exponentialfourier}
\psi_{\sigma_2}(\xi):=\frac{\Gamma(\frac{d+1}2)}{\pi^{\frac{d+1}2}}\frac{\sigma_2^d}{(1+\sigma_2^2\|\xi\|^2)^{\frac{d+1}2}},\ \ \xi\in\bR^d.
\end{equation}
Here, $\Gamma$ denotes the Gamma function and the Fourier transform is identified by the Poisson kernel (see, for example, \cite{Stein}, page 61).

\item[--] the inverse multiquadrics
\begin{equation}\label{multiquadric}
M_\beta(x,y):=\frac1{(1+\|x-y\|^2)^\beta}=\int_{\bR^d}e^{i(x-y,\xi)}m_\beta(\xi)d\xi,\ \ x,y\in\bR^d,\ \ \beta>0,
\end{equation}
where
\begin{equation}\label{multiquadricfourier}
m_\beta(\xi):=\frac1{(2\sqrt{\pi})^d}\frac1{\Gamma(\beta)}\int_0^\infty t^{\beta-\frac d2-1}\exp\left(-\frac{\|\xi\|^2}{4t}-t\right)dt,\ \ \xi\in\bR^d.
\end{equation}
This formulation can be obtained by combining Theorem 7.15 in \cite{Wendland} and the Fourier transform of the Gaussian function.

\item[--] the B-spline kernel
\begin{equation}\label{bspline}
B_p(x,y):=\prod_{j=1}^d B_p(x_j-y_j)=\int_{\bR^d}e^{i(x-y,\xi)}b_p(\xi)d\xi,\ \ x,y\in\bR^d,\ p\in2\bN,
\end{equation}
where $B_p$ denotes the $p$-th order cardinal B-spline, and with $\sinc_{\frac12}(t):=\frac{\sin(\frac{t}2)}{\frac{t}2}$, $t\in\bR$,
$$
b_p(\xi):=\frac1{(2\pi)^d}\prod_{j=1}^d(\sinc_{\frac12}(\xi_j))^p,\ \ \xi\in\bR^d.
$$

\item[--] the ANOVA kernel
\begin{equation}\label{ANOVA}
A_\tau(x,y):=\sum_{j=1}^d \exp\biggl(-\frac{|x_j-y_j|^2}{\tau}\biggr)=\int_{\bR^d}e^{i(x-y,\xi)}a_\tau(\xi)
d\xi,\ \ x,y\in\bR^d,\ \tau>0,
\end{equation}
where
$$
a_\tau(\xi):=\frac{\sqrt{\tau}}{2\sqrt{\pi}}\left(\sum_{j=1}^d\exp(-\frac{\tau \xi_j^2}4)\right),\ \ \xi\in\bR^d.
$$
\end{itemize}

Among those kernels, the Gaussian kernel, the $\ell^2$-norm exponential kernel, and the inverse multiquadrics are radial basis functions.
We also give their representation by the Laplace transform below:

\begin{itemize}

\item[--] the Gaussian kernel
\begin{equation}\label{gaussianRBF}
G_\gamma(x,y)=\exp\biggl(-\frac{\|x-y\|^2}{\gamma}\biggr)=\int_0^\infty e^{-\|x-y\|^2t}d\delta_{\gamma^{-1}}(t),
\end{equation}
where $\delta_t$ denotes the unit measure supported at the singleton $\{t\}$.

\item[--] the $\ell^2$-norm exponential kernel
\begin{equation}\label{l2exponentialRBF}
\cE_{\sigma_2}(x,y)=\exp\biggl(-\frac{\|x-y\|}{\sigma_2}\biggr)=\frac1{2\sigma_2\sqrt{\pi}}\int_0^\infty e^{-\|x-y\|^2t}\exp(-\frac1{4\sigma_2^2t})\frac1{t^{3/2}}dt,
\ \ x,y\in\bR^d.
\end{equation}
This equation is derived from the identity (see \cite{Stein}, page 61) that
$$
e^{-r} = \frac{1}{\sqrt{\pi}}\int_0^\infty e^{-r^2/4s} \frac{e^{-s}}{\sqrt{s}}\; ds, \quad r>0.
$$

\item[--] the inverse multiquadrics (see \cite{Wendland}, page 95)
\begin{equation}\label{multiquadricRBF}
M_\beta(x,y)=\frac1{(1+\|x-y\|^2)^\beta}=\frac1{\Gamma(\beta)}\int_0^\infty e^{-\|x-y\|^2t}t^{\beta-1}e^{-t}dt,\ \ x,y\in\bR^d.
\end{equation}
\end{itemize}

As a straightforward application of Corollary \ref{chatranslationcontinuous}, we have the following inclusion relations between the RKHS of kernels of the same kind.

\begin{prop}\label{inclusionsameclass}
The following statements hold true:
\begin{enumerate}[(1)]
\item For $0<\gamma_1<\gamma_2$, $\cH_{G_{\gamma_2}}\subseteq\cH_{G_{\gamma_1}}$ with $$
    \lambda(G_{\gamma_2},G_{\gamma_1})=\left(\frac{\gamma_2}{\gamma_1}\right)^{\frac d2},
    $$
    but $\cH_{G_{\gamma_1}}\nsubseteq\cH_{G_{\gamma_2}}$.

    \item For $0<\sigma_1<\sigma_2$, $\cH_{E_{\sigma_1}}=\cH_{E_{\sigma_2}}$ with
        $$
        \lambda(E_{\sigma_1},E_{\sigma_2})= \lambda(E_{\sigma_2},E_{\sigma_1})=\left(\frac{\sigma_2}{\sigma_1}\right)^d.
        $$

        \item For $0<\sigma_1<\sigma_2$, $\cH_{\cE_{\sigma_1}}=\cH_{\cE_{\sigma_2}}$ with
        $$
        \lambda(\cE_{\sigma_1},\cE_{\sigma_2})=\frac{\sigma_2}{\sigma_1},\ \ \lambda(\cE_{\sigma_2},\cE_{\sigma_1})=\left(\frac{\sigma_2}{\sigma_1}\right)^d.
        $$

        \item For $p,q\in 2\bN$ with $p<q$, $\cH_{B_q}\subseteq\cH_{B_p}$ with $\lambda(B_q,B_p)=1$, but $\cH_{B_p}\nsubseteq\cH_{B_q}$.

            \item For $0<\tau_1<\tau_2$, $\cH_{A_{\tau_2}}\subseteq\cH_{A_{\tau_1}}$ with $\lambda(A_{\tau_2},A_{\tau_1})=\sqrt{\frac{\tau_2}{\tau_1}}$, but $\cH_{A_{\tau_1}}\nsubseteq\cH_{A_{\tau_2}}$.
\end{enumerate}
\end{prop}

The inclusion relation for the RKHS of two inverse multiquadrics is more involved and is separated below.

\begin{theorem}\label{twomultiquadrics}
Let $\beta_1,\beta$ be two distinct positive constants. There holds $\cH_{M_{\beta_1}}\subseteq\cH_{M_{\beta_2}}$ if and only if $\frac d2<\beta_1<\beta_2$.
\end{theorem}
\begin{proof}
Suppose first that $\beta_1>\beta_2$. By the same technique used in Theorem 6.13, \cite{Wendland} and equation (\ref{multiquadricfourier}), one obtains for all $\beta>0$ that
\begin{equation}\label{twomultiquadricseq1}
m_\beta(\xi)=\frac{2^{1-\beta}}{(\sqrt{2\pi})^d\Gamma(\beta)}\|\xi\|^{\beta-\frac d2}K_{\beta-\frac d2}(\|\xi\|),\ \ \xi\ne 0,
\end{equation}
where $K_\nu$, $\nu\in\bR$ is the modified Bessel functions defined by
$$
K_\nu(r):=\int_0^\infty e^{-r\cosh t}\cosh(\nu t)dt,\ \ r>0.
$$
We use the estimates (see, \cite{Wendland}, pages 52-53) about $K_\nu$ that there exists a constant $C_\nu$ depending on $\nu$ only such that
\begin{equation}\label{twomultiquadricseq2}
K_\nu(r)\ge C_\nu \frac{e^{-r}}{\sqrt{r}},\ \ r\ge1
\end{equation}
and that
\begin{equation}\label{twomultiquadricseq3}
K_\nu(r)\le\sqrt{2\pi}\frac{e^{-r}}{\sqrt{r}}\exp\left(\frac{|\nu|^2}{2r}\right),\ \ r>0.
\end{equation}
Combining equations (\ref{twomultiquadricseq1}), (\ref{twomultiquadricseq2}), and (\ref{twomultiquadricseq3}), we obtain for $\beta_1>\beta_2$ that
\begin{equation}\label{twomultiquadricseq4}
\frac{m_{\beta_1}(\xi)}{m_{\beta_2}(\xi)}\ge\frac{C_{\beta_1-\frac d2} 2^{\beta_2-\beta_1}\Gamma(\beta_2)}{\sqrt{2\pi}\Gamma(\beta_1)}\|\xi\|^{\beta_1-\beta_2}\exp\left(-\frac{|\beta_2-\frac d2|^2}{2\|\xi\|}\right),\ \ \ \|\xi\|\ge 1.
\end{equation}
Since the right hand side above goes to infinity as $\|\xi\|\to\infty$, we get by Corollary \ref{chatranslationcontinuous} that $\cH_{M_{\beta_1}}\nsubseteq\cH_{M_{\beta_2}}$ when $\beta_1>\beta_2$.

By monotone convergence theorem, we have by equation (\ref{multiquadricfourier}) for all $\beta>0$ that
\begin{equation}\label{twomultiquadricseq5}
\lim_{\xi\to 0}m_\beta(\xi)=\left\{
\begin{array}{ll}
+\infty,&\mbox{ if }\beta\le \frac d2,\\
\frac1{(2\sqrt{\pi})^d}\frac{\Gamma(\beta-\frac d2)}{\Gamma(\beta)}<+\infty,&\mbox{ if }\beta>\frac d2.
\end{array}
\right.
\end{equation}
Therefore, if $\beta_1\le \frac d2<\beta_2$ then $m_{\beta_1}(\xi)/m_{\beta_2}(\xi)$ is unbounded on a neighborhood of the origin. As a consequence, $\cH_{M_{\beta_1}}\nsubseteq\cH_{M_{\beta_2}}$ in this case.

Suppose that $\frac d2<\beta_1<\beta_2$. Then by (\ref{twomultiquadricseq5}), $m_{\beta_1}(\xi)/m_{\beta_2}(\xi)$ is bounded on a neighborhood of the origin. Also, by (\ref{twomultiquadricseq4}),
$$
\lim_{\|\xi\|\to\infty}\frac{m_{\beta_1}(\xi)}{m_{\beta_2}(\xi)}=0.
$$
As $m_{\beta_1}(\xi)/m_{\beta_2}(\xi)$ is continuous on $\bR^d\setminus\{0\}$, it is hence bounded therein. By Corollary \ref{chatranslationcontinuous}, $\cH_{M_{\beta_1}}\subseteq\cH_{M_{\beta_2}}$ when $\frac d2<\beta_1<\beta_2$.

We now discuss the last case that $\beta_1<\beta_2\le \frac d2$. We shall show that in this case $\cH_{M_{\beta_1}}\nsubseteq\cH_{M_{\beta_2}}$ by proving that $m_{\beta_1}(\xi)/m_{\beta_2}(\xi)$ is unbounded on a neighborhood of the origin. To this end, let $\|\xi\|\le 1$ and use the change of variables $t=\|\xi\|^2s$ in (\ref{multiquadricfourier}) to get that
\begin{equation}\label{twomultiquadricseq7}
\frac{m_{\beta_1}(\xi)}{m_{\beta_2}(\xi)}=\|\xi\|^{2(\beta_1-\beta_2)}\frac{\Gamma(\beta_2)}{\Gamma(\beta_1)}\frac{\displaystyle{\int_0^\infty s^{\beta_1-\frac d2-1}\exp\left(-\frac1{4s}-\|\xi\|^2s\right)ds}}{\displaystyle{\int_0^\infty s^{\beta_2-\frac d2-1}\exp\left(-\frac1{4s}-\|\xi\|^2s\right)ds}}.
\end{equation}
Thus, if $\beta_2<\frac d2$ then we have for $\|\xi\|\le 1$ that
$$
\frac{m_{\beta_1}(\xi)}{m_{\beta_2}(\xi)}\ge \|\xi\|^{2(\beta_1-\beta_2)}\frac{\Gamma(\beta_2)}{\Gamma(\beta_1)}\frac{\displaystyle{\int_0^\infty s^{\beta_1-\frac d2-1}\exp\left(-\frac1{4s}-s\right)ds}}{\displaystyle{\int_0^\infty s^{\beta_2-\frac d2-1}\exp\left(-\frac1{4s}\right)ds}}.
$$
The right hand side above is unbounded when $\|\xi\|\to 0$. When $\beta_2=\frac d2$, we estimate that
$$
\int_0^\infty s^{\beta_2-\frac d2-1}\exp\left(-\frac1{4s}-\|\xi\|^2s\right)ds\le \int_0^1\frac 1s e^{-\frac1{4s}}ds+\int_{1}^\infty \frac1se^{-\|\xi\|^2s}ds.
$$
A change of variables $\|\xi\|^2s=t$ then yields for $\|\xi\|<1$ that
$$
\int_{1}^\infty \frac1se^{-\|\xi\|^2s}ds=\int_{\|\xi\|^2}^\infty \frac1te^{-t}dt\le \int_{\|\xi\|^2}^1 \frac1tdt+\int_1^\infty \frac{e^{-t}}tdt=-2\ln(\|\xi\|)+\int_1^\infty \frac{e^{-t}}tdt.
$$
Combining the above two equations with (\ref{twomultiquadricseq7}) yields that
$$
\frac{m_{\beta_1}(\xi)}{m_{\beta_2}(\xi)}\ge \|\xi\|^{2(\beta_1-\beta_2)}\frac{\Gamma(\beta_2)}{\Gamma(\beta_1)}\frac{\displaystyle{\int_0^\infty s^{\beta_1-\frac d2-1}\exp\left(-\frac1{4s}-s\right)ds}}{\displaystyle{\int_0^1\frac 1s e^{-\frac1{4s}}ds+\int_1^\infty \frac{e^{-t}}tdt}-2\ln(\|\xi\|)},\ \ \|\xi\|<1.
$$
The right hand side above goes to infinity as $\|\xi\|\to0$. The proof is complete.
\end{proof}

The main purpose of this section is to explore the inclusion relationships among RKHS of different kinds of translation invariant kernels given above. We present the results in the form of a table.

\begin{theorem}
Let $p\in 2\bN$ and $\gamma,\sigma_1,\sigma_2,\beta,\tau$ be positive constants. The following inclusion relations of RKHS hold true.
$$
\begin{array}{|c|c|c|c|c|c|c|}\hline
&\cH_{B_p}&\cH_{G_\gamma}&\cH_{E_{\sigma_1}}&\cH_{\cE_{\sigma_2}}&\cH_{M_\beta}&\cH_{A_\tau}\\\hline
\cH_{B_p}&=&\nsubseteq&\subseteq &\subseteq\mbox{ iff }p\ge d+1&\nsubseteq&\nsubseteq\\\hline
\cH_{G_\gamma}&\nsubseteq&=&\subseteq&\subseteq&\subseteq&\subseteq\mbox{ iff }\gamma\ge\tau\\\hline
\cH_{E_{\sigma_1}}&\nsubseteq&\nsubseteq&=&\nsubseteq\mbox{ if }d\ge2&\nsubseteq&\nsubseteq\\\hline
\cH_{\cE_{\sigma_2}}&\nsubseteq&\nsubseteq&\nsubseteq\mbox{ if }d\ge2&=&\nsubseteq&\nsubseteq\\\hline
\cH_{M_\beta}&\nsubseteq&\nsubseteq&\subseteq\mbox{ iff }\beta>\frac d2&\subseteq\mbox{ iff }\beta>\frac d2&=&\nsubseteq\\\hline
\cH_{A_\tau}&\nsubseteq&\nsubseteq\mbox{ if }d\ge2&\nsubseteq&\nsubseteq&\nsubseteq&=\\\hline
\end{array}
$$
\end{theorem}

We break the task of proving this result into several steps as follows.
\begin{enumerate}[(i)]
\item For any dimension $d\in\bN$ and parameters $p\in2\bN$, $\gamma,\tau>0$, $\cH_{B_p}\nsubseteq \cH_{K}$ and $\cH_K\nsubseteq \cH_{B_p}$
for $K=G_\gamma$ or $K=A_\tau$.

\begin{proof}
We first discuss the case when $K=G_\gamma$. It is clear that $b_p/g_\gamma$ is unbounded on $\bR^d$. By Corollary \ref{chatranslationcontinuous},
$\cH_{B_p}\nsubseteq \cH_{G_\gamma}$. On the other hand, $b_p$ possesses zeros on $\bR^d$ while $g_\gamma$ is everywhere positive. As they
are both continuous, there does not exist a positive constant $\lambda>0$ such that $g_\gamma(\xi)\le \lambda^2 b_p(\xi)$ for
almost every $\xi\in\bR^d$. As a consequence, we obtain by Corollary \ref{chatranslationcontinuous} that $\cH_{G_\gamma}\nsubseteq\cH_{B_p}$.
The other case when $K=A_\tau$ can be handled in a similar way.
\end{proof}

\item For any $d\in\bN$, $\sigma_2>0$ and $p\in2\bN$, $\cH_{\cE_{\sigma_2}}\nsubseteq\cH_{B_p}$. There holds
$\cH_{B_p}\subseteq\cH_{\cE_{\sigma_2}}$ if and only if $p\ge d+1$, in which case
\begin{equation}\label{estimatebpl2exponential}
\lambda(B_p,\,\cE_{\sigma_2})\le \frac{2^{p-d}}{\sigma_2^d\pi^{\frac{d-1}2}} \frac{(1+\sigma_2^2d)^{\frac{d+1}2}}{\Gamma(\frac{d+1}2)}.
\end{equation}
\begin{proof}
The function $\psi_{\sigma_2}$ in (\ref{l2exponential}) is continuous and positive everywhere on $\bR^d$.
By arguments used before, $\cH_{\cE_{\sigma_2}}\nsubseteq \cH_{B_p}$. Assume that $p<d+1$.
We choose $\xi_1=(2n+1)\pi$ and $\xi_j=0$ for $j\ge2$ to get that $b_p(\xi)=O(n^{-p})$ while $\psi_{\sigma_2}(\xi)=O(n^{-(d+1)})$
as $n$ tends to infinity. Therefore, $b_p/\psi_{\sigma_2}$ is unbounded on $\bR^d$, implying that $\cH_{B_p}\nsubseteq \cH_{\cE_{\sigma_2}}$.

Suppose that $p\ge d+1$. If $\|\xi\|_\infty:=\max\{|\xi_j|:j\in\bN_d\}\le1$ then
$$
b_p(\xi)\le \frac1{(2\pi)^d}, \ \ \psi_{\sigma_2}(\xi)\ge \frac{\Gamma(\frac{d+1}2)}{\pi^{\frac{d+1}2}}\frac{\sigma_2^d}{(1+\sigma_2^2d)^{\frac{d+1}2}}.
$$
It follows that
\begin{equation}\label{Bpl2exponentialestimateeq1}
\frac{b_p(\xi)}{\psi_{\sigma_2}(\xi)}\le \frac1{(2\sigma_2\pi)^d} \frac{\pi^{\frac{d+1}2}}{\Gamma(\frac{d+1}2)}
(1+\sigma_2^2d)^{\frac{d+1}2},\quad
\|\xi\|_\infty\le 1.
\end{equation}
When $\|\xi\|_\infty\ge 1$,
$$
b_p(\xi)\le \frac1{(2\pi)^d}\frac{2^p}{\|\xi\|_\infty^p},
$$
which implies by $p\ge d+1$ that for $\|\xi\|_\infty\ge 1$,
$$
\begin{array}{rl}
\displaystyle{\frac{b_p(\xi)}{\psi_{\sigma_2}(\xi)}}&\displaystyle{\le\frac{2^p}{(2\pi)^d} \frac{\pi^{\frac{d+1}2}}{\Gamma(\frac{d+1}2)}
\frac{(1+\sigma_2^2d\|\xi\|_\infty^2)^{\frac{d+1}2}}{\sigma_2^d\|\xi\|_\infty^p}
\le \frac{2^p}{(2\sigma_2\pi)^d} \frac{\pi^{\frac{d+1}2}}
{\Gamma(\frac{d+1}2)}\biggl(\frac1{\|\xi\|_\infty^2}+\sigma_2^2d\biggr)^{\frac{d+1}2}}\\
&\displaystyle{\le\frac{2^p}{(2\sigma_2\pi)^d} \frac{\pi^{\frac{d+1}2}}{\Gamma(\frac{d+1}2)}(1+\sigma_2^2d)^{\frac{d+1}2}}.
\end{array}
$$
By Corollary \ref{chatranslationcontinuous}, the above inequality together with (\ref{Bpl2exponentialestimateeq1}) proves (\ref{estimatebpl2exponential}).
\end{proof}

\item For any $d\in\bN$, $\sigma_1>0$ and $p\in2\bN$, $\cH_{E_{\sigma_1}}\nsubseteq\cH_{B_p}$. There holds
$\cH_{B_p}\subseteq\cH_{E_{\sigma_1}}$ and
\begin{equation}\label{estimatebpl1exponential}
\lambda(B_p,\,E_{\sigma_1})\le 2^{d}\left(\sigma_1+\frac1{\sigma_1}\right)^{d}.
\end{equation}
\begin{proof}
The relation $\cH_{E_{\sigma_1}}\nsubseteq\cH_{B_p}$ follows from that $\varphi_{\sigma_1}$ is positive and continuous everywhere on $\bR^d$.
Using an estimate method similar to that in (ii), we get that
$$
(\sinc_{\frac12}(t))^p(1+\sigma_1^2t^2)\le (\sinc_{\frac12}(t))^2(1+\sigma_1^2t^2)\le 4(1+\sigma_1^2)\mbox{ for all }t\in\bR,
$$
which combined with the explicit form of $b_p$ and $\varphi_{\sigma_1}$ leads to (\ref{estimatebpl1exponential}).
\end{proof}

\item For any $d\in\bN$, $\sigma_1>0$ and $\gamma>0$, $\cH_{E_{\sigma_1}}\nsubseteq\cH_{G_\gamma}$. There holds
$\cH_{G_\gamma}\subseteq\cH_{E_{\sigma_1}}$ and
\begin{equation}\label{estimategaussianl1exponential}
\lambda(G_\gamma,\,E_{\sigma_1})\le \biggl(\max(1,\frac{4\sigma_1^2}\gamma)\,\frac{\sqrt{\gamma\pi}}{2\sigma_1}\biggr)^{d}.
\end{equation}
\begin{proof}
It is clear that $\varphi_{\sigma_1}/g_\gamma$ is unbounded on $\bR^d$. By Corollary \ref{chatranslationcontinuous},
$\cH_{E_{\sigma_1}}\nsubseteq\cH_{G_\gamma}$. On the other hand, one has that
$$
\frac{g_\gamma(\xi)}{\varphi_{\sigma_1}(\xi)}= \biggl(\frac{\sqrt{\gamma\pi}}{2\sigma_1}\biggr)^d
\exp(-\frac{\gamma\|\xi\|^2}4)\prod_{j=1}^d (1+\sigma_1^2\xi_j^2),\ \ \xi\in\bR^d,
$$
which together with the observation that
$$
(1+\sigma_1^2\xi_j^2)\le \max(1,\frac{4\sigma_1^2}\gamma)\exp\biggl(\frac{\gamma\xi_j^2}4\biggr),\ \ \xi_j\in\bR
$$
proves (\ref{estimategaussianl1exponential}).
\end{proof}

\item For any $d\in\bN$, $\sigma_1>0$ and $\gamma>0$, $\cH_{\cE_{\sigma_2}}\nsubseteq\cH_{G_\gamma}$. There holds
$\cH_{G_\gamma}\subseteq\cH_{\cE_{\sigma_2}}$ and
\begin{equation}\label{estimategaussianl2exponential}
\lambda(G_\gamma,\,\cE_{\sigma_2})\le \biggl(\max(1,\frac{(2d+2)\sigma_2^2}\gamma)\biggr)^{\frac{d+1}2}
\,\biggl(\frac{\sqrt{\gamma}}{2\sigma_2}\biggr)^d\frac{\pi^{\frac{d-1}2}}{\Gamma(\frac{d+1}2)}.
\end{equation}
However, $\lambda(G_\gamma,\,\cE_{\sigma_2})$ does not have a common upper bound as $d$ varies on $\bN$.

\begin{proof}
As $\psi_{\sigma_2}/g_\gamma$ is clearly unbounded on $\bR^d$, $\cH_{\cE_{\sigma_2}}\nsubseteq\cH_{G_\gamma}$. We then estimate that for all
$\xi\in\bR^d$,
$$
(1+\sigma_2^2\|\xi\|^2)^{\frac{d+1}2}\le \biggl(\max(1,\frac{(2d+2)\sigma_2^2}\gamma)\biggr)^{\frac{d+1}2}
(1+\frac{\gamma\|\xi\|^2}{2d+2})^{\frac{d+1}2}\le
\biggl(\max(1,\frac{(2d+2)\sigma_2^2}\gamma)\biggr)^{\frac{d+1}2}\exp(\frac{\gamma\|\xi\|^2}4),
$$
which immediately implies that $g_\gamma(\xi)/\psi_{\sigma_2}(\xi)$ is bounded by the right hand side of (\ref{estimategaussianl2exponential}).
Equation (\ref{estimategaussianl2exponential}) now follows from Corollary \ref{chatranslationcontinuous}.

To prove the third claim, we use the Laplace transform representations (\ref{gaussianRBF}) and (\ref{l2exponentialRBF}). One observes that
the Gaussian kernel $G_\gamma$ corresponds to the delta measure $\delta_{\gamma^{-1}}$, which is singular
 with respect to the Lebesgue measure while $\cE_{\sigma_2}$ is represented by
the Borel measure
$$
\frac1{2\sigma_2\sqrt{\pi}}\exp(-\frac1{4\sigma_2^2t})\frac1{t^{3/2}}dt,
$$
which is absolutely continuous with respect to the Lebesgue measure. Thus, $\delta_{\gamma^{-1}}$ is not absolutely continuous with respect to
the above measure. By Proposition \ref{characterizeRBF}, $\lambda(G_\gamma,\,\cE_{\sigma_2})$ does not have a common upper bound as the dimension
$d$ varies on $\bN$.
\end{proof}

\item For any $d\ge 2$, $\sigma_1>0$ and $\sigma_2>0$, $\cH_{\cE_{\sigma_2}}\nsubseteq\cH_{E_{\sigma_1}}$
and $\cH_{E_{\sigma_1}}\nsubseteq \cH_{\cE_{\sigma_2}}$.

\begin{proof}
We first let $\xi_1=n$ and $\xi_j=0$ for $j\ge2$ to get that $\varphi_{\sigma_1}(\xi)=O(n^{-2})$ and
$\psi_{\sigma_2}(\xi)=O(n^{-(d+1)})$ as $n$ tends to infinity. As $d\ge2$, $\varphi_{\sigma_1}(\xi)/\psi_{\sigma_2}(\xi)$ is
unbounded on $\bR^d$, implying that $\cH_{E_{\sigma_1}}\nsubseteq \cH_{\cE_{\sigma_2}}$.
The choice $\xi_j=n$ for all $j\in\bN_d$ yields that $\varphi_{\sigma_1}(\xi)=O(n^{-2d})$ and
$\psi_{\sigma_2}(\xi)=O(n^{-(d+1)})$ as $n\to\infty$. Therefore, $\psi_{\sigma_2}(\xi)/\varphi_{\sigma_1}(\xi)$
is unbounded on $\bR^d$. It implies that $\cH_{\cE_{\sigma_2}}\nsubseteq\cH_{E_{\sigma_1}}$.
\end{proof}

\item For any $d\ge 2$, $\sigma_1,\sigma_2,\tau>0$, $\cH_{A_\tau}\nsubseteq\cH_{K}$
and $\cH_{K}\nsubseteq \cH_{A_\tau}$ for either $K=E_{\sigma_1}$ or $K=\cE_{\sigma_2}$.

\begin{proof}
We discuss $K=E_{\sigma_1}$ only as the other case can be dealt with similarly. Choosing $\xi_j=n$ for all $j\in\bN_d$ yields that
$\varphi_{\sigma_1}(\xi)/a_\tau(\xi)\to\infty$ as $n\to\infty$. The other choice
$\xi_1=n$ and $\xi_j=0$ for $j\ge2$ tells that $a_\tau(\xi)/\varphi_{\sigma_1}(\xi)\to\infty$ as $n\to\infty$. Therefore, neither
$\varphi_{\sigma_1}/a_\tau$ nor $a_\tau/\varphi_{\sigma_1}$ is bounded on $\bR^d$. The result now follows from Corollary \ref{chatranslationcontinuous}.
\end{proof}

\item For any $d\ge2$, $\gamma,\tau>0$, $\cH_{A_\tau}\nsubseteq \cH_{G_\gamma}$. There holds $\cH_{G_\gamma}\subseteq \cH_{A_\tau}$
if and only if $\gamma\ge \tau$, in which case
\begin{equation}\label{lambdagaussiananova}
\lambda(G_\gamma,A_\tau)= \frac{\sqrt{\gamma}^d}{d\sqrt{\tau}(2\sqrt{\pi})^{d-1}}.
\end{equation}
\begin{proof}
That $\cH_{A_\tau}\nsubseteq \cH_{G_\gamma}$ can be proved in a way similar to that in (vii). If $\gamma<\tau$ then we set $\xi_j=n$ for all
$j\in\bN_d$ to see that $g_\gamma(\xi)/a_\tau(\xi)\to \infty$ as $n\to\infty$. Thus, $\cH_{G_\gamma}\nsubseteq \cH_{A_\tau}$ in this case.
Suppose that $\gamma\ge \tau$. We get for all $\xi\in\bR^d$ that
$$
\frac{g_\gamma(\xi)}{a_\tau(\xi)}=\frac{\sqrt{\gamma}^d}{\sqrt{\tau}(2\sqrt{\pi})^{d-1}}
\frac{\exp(-\frac{\gamma\|\xi\|^2}4)}{\sum_{j=1}^d \exp(-\frac{\tau|\xi_j|^2}4)},
$$
which together with the observation that
$$
\exp(-\frac{\gamma\|\xi\|^2}4)\le \exp(-\frac{\tau|\xi_j|^2}4) \mbox{ for all }j\in\bN_d
$$
implies that
$$
\frac{g_\gamma(\xi)}{a_\tau(\xi)}\le \frac{\sqrt{\gamma}^d}{d\sqrt{\tau}(2\sqrt{\pi})^{d-1}}\mbox{ for all }\xi\in\bR^d.
$$
As the equality is achieved at $\xi=0$, we obtain (\ref{lambdagaussiananova}).
\end{proof}

\item For any $d\in\bN$, $\sigma_2,\beta>0$, $\cH_{\cE_{\sigma_2}}\nsubseteq \cH_{M_\beta}$. There holds $\cH_{M_\beta}\subseteq \cH_{\cE_{\sigma_2}}$ if and only if $\beta>\frac d2$.

\begin{proof}
By (\ref{l2exponentialfourier}) and (\ref{multiquadricfourier}), we have for all $\xi\in\bR^d$ that
\begin{equation}\label{estimatemultil2exponentialeq1}
\frac{m_\beta(\xi)}{\psi_{\sigma_2}(\xi)}=\frac1{(2\sigma_2\sqrt{\pi})^d}\frac{\pi^{\frac{d+1}2}}{\Gamma(\beta)\Gamma(\frac{d+1}2)}
\int_0^\infty t^{\beta-\frac d2-1}(1+\sigma_2^2\|\xi\|^2)^{\frac{d+1}2}\exp\biggl(-\frac{\|\xi\|^2}{4t}-t\biggr)dt.
\end{equation}
Note that when $\|\xi\|\ge1$,
$$
\exp\biggl(-\frac{\|\xi\|^2}{4t}-t\biggr)=\exp\biggl(-\frac{\|\xi\|^2}{8t}-\frac t2\biggr)\exp\biggl(-\frac{\|\xi\|^2}{8t}-\frac t2\biggr)\le \exp(-\frac{\|\xi\|}2)\exp\biggl(-\frac{1}{8t}-\frac t2\biggr).
$$
Thus, for $\|\xi\|\ge1$
$$
\int_0^\infty t^{\beta-\frac d2-1}(1+\sigma_2^2\|\xi\|^2)^{\frac{d+1}2}\exp\biggl(-\frac{\|\xi\|^2}{4t}-t\biggr)dt\le \frac{(1+\sigma_2^2\|\xi\|^2)^{\frac{d+1}2}}{\exp(\frac{\|\xi\|}2)}\int_0^\infty t^{\beta-\frac d2-1}\exp\biggl(-\frac{1}{8t}-\frac t2\biggr)dt.
$$
We hence get that $m_\beta(\xi)/\psi_{\sigma_2}(\xi)\to 0$ as $\|\xi\|\to\infty$. It implies that $\cH_{\cE_{\sigma_2}}\nsubseteq \cH_{M_\beta}$.

To prove the rest of the claims, one first sees by the Lebesgue dominated convergence theorem that $m_\beta/\psi_{\sigma_2}$ is continuous on $\bR^d\setminus\{0\}$. We also have that $m_\beta(\xi)/\psi_{\sigma_2}(\xi)\to 0$ as $\|\xi\|\to\infty$. For these two reasons, $m_\beta/\psi_{\sigma_2}$ is essentially bounded on $\bR^d$ if and only if it is bounded on a neighborhood of the origin. If $\beta>\frac d2$, we observe that when $\|\xi\|\le 1$,
$$
\int_0^\infty t^{\beta-\frac d2-1}(1+\sigma_2^2\|\xi\|^2)^{\frac{d+1}2}\exp\biggl(-\frac{\|\xi\|^2}{4t}-t\biggr)dt
\le (1+\sigma_2^2)^{\frac{d+1}2}\int_0^\infty t^{\beta-\frac d2-1}e^{-t}dt<+\infty,
$$
which implies that $m_\beta/\psi_{\sigma_2}$ is essentially bounded on $\bR^d$ when $\beta>\frac d2$. We hence get by Corollary \ref{chatranslationcontinuous} that $\cH_{M_\beta}\subseteq \cH_{\cE_{\sigma_2}}$ in this case. When $\beta\le \frac d2$, by the monotone convergence theorem,
$$
\lim_{\|\xi\|\to 0}\int_0^\infty t^{\beta-\frac d2-1}(1+\sigma_2^2\|\xi\|^2)^{\frac{d+1}2}\exp\biggl(-\frac{\|\xi\|^2}{4t}-t\biggr)dt=
\int_0^\infty t^{\beta-\frac d2-1}e^{-t}dt=+\infty.
$$
It follows from the above equation that $\cH_{M_\beta}\nsubseteq \cH_{\cE_{\sigma_2}}$ when $\beta\le \frac d2$.\end{proof}

\item  For any $d\in\bN$, $\sigma_1,\beta>0$, $\cH_{E_{\sigma_1}}\nsubseteq \cH_{M_\beta}$. There holds $\cH_{M_\beta}\subseteq \cH_{E_{\sigma_1}}$ if and only if $\beta>\frac d2$.

    \begin{proof} The proof is similar to that for (ix).
    \end{proof}

\item For any $d\in\bN$, $p\in2\bN$, $\beta>0$, $\cH_{B_p}\nsubseteq \cH_{M_\beta}$ and $\cH_{M_\beta}\nsubseteq \cH_{B_p}$.

    \begin{proof}
    As $m_\beta$ is positive and continuous on $\bR^d\setminus\{0\}$ while $b_p$ has zeros on $\bR^d\setminus\{0\}$, $\cH_{M_\beta}\nsubseteq \cH_{B_p}$. That $\cH_{B_p}\nsubseteq \cH_{M_\beta}$ can be proved by arguments similar to those in (ix).
    \end{proof}

\item For any $d\in\bN$, $\gamma,\beta>0$, $\cH_{M_\beta}\nsubseteq \cH_{G_\gamma}$ but $\cH_{G_\gamma}\subseteq \cH_{M_\beta}$. The quantity $\lambda(G_\gamma,M_\beta)$ does not have a common upper bound as $d$ varies on $\bN$.

    \begin{proof}
    We start with the observation that
    $$
    \begin{array}{rl}
    \displaystyle{\frac{m_\beta(\xi)}{g_\gamma(\xi)}}&\displaystyle{=\frac{1}{\Gamma(\beta)\gamma^{\frac d2}}\int_0^\infty t^{\beta-\frac d2-1}\exp\biggl(\frac{\gamma \|\xi\|^2}4\biggr)\exp\left(-\frac{\|\xi\|^2}{4t}-t\right)dt}\\
    &\ge \displaystyle{\frac{1}{\Gamma(\beta)\gamma^{\frac d2}}\int_{\frac2\gamma}^\infty t^{\beta-\frac d2-1}\exp\biggl(\frac{\gamma \|\xi\|^2}4-\frac{\|\xi\|^2}{4t}\biggr)e^{-t}dt}\\
    &\displaystyle{\ge \frac{1}{\Gamma(\beta)\gamma^{\frac d2}}\exp\biggl(\frac{\gamma \|\xi\|^2}8\biggr)\int_{\frac2\gamma}^\infty t^{\beta-\frac d2-1}e^{-t}dt}.
    \end{array}
    $$
    Therefore, $m_\beta(\xi)/g_\gamma(\xi)$ tends to infinity as $\|\xi\|\to\infty$. Consequently, $\cH_{M_\beta}\nsubseteq \cH_{G_\gamma}$.

    We also notice by the monotone convergence theorem that
    $$
    \lim_{\|\xi\|\to 0}\frac{m_\beta(\xi)}{g_\gamma(\xi)}=\frac{1}{\Gamma(\beta)\gamma^{\frac d2}}\int_0^\infty t^{\beta-\frac d2-1}e^{-t}dt>0.
    $$
    As $\frac{m_\beta(\xi)}{g_\gamma(\xi)}$ is continuous and positive everywhere on $\bR^d\setminus\{0\}$, the above two estimates imply that there exists some positive constant $\lambda$ such that
    $$
    \frac{m_\beta(\xi)}{g_\gamma(\xi)}\ge \lambda\mbox{ for all }\xi\in\bR^d\setminus\{0\}.
    $$
    We hence conclude that $\cH_{G_\gamma}\subseteq \cH_{M_\beta}$. Recall (\ref{gaussianRBF}) and (\ref{multiquadricRBF}). Since $G_\gamma$ and $M_\beta$ are respectively represented by measures singular and absolutely continuous with respect to the Lebesgue measure, $\lambda(G_\gamma,M_\beta)$ does not have a common upper bound for $d\in\bN$.
    \end{proof}

\item For any $d\in\bN$, $\tau,\beta>0$, $\cH_{A_\tau}\nsubseteq \cH_{M_\beta}$ and $\cH_{M_\beta}\nsubseteq \cH_{A_\tau}$.

\begin{proof}
Firstly, we see for the choice $\xi_1=n$, $\xi_j=0$, $j\ge 2$ that
$$
\lim_{n\to\infty}m_\beta(\xi)=0\mbox{ while }\lim_{n\to\infty}a_\tau(\xi)=\frac{(d-1)\sqrt{\tau}}{2\sqrt{\pi}}.
$$
As a result, $\cH_{A_\tau}\nsubseteq \cH_{M_\beta}$. Secondly, arguments similar to those in (xii) shows that for the choice $\xi_j=n$, $j\in\bN_d$
$$
\lim_{n\to\infty}\frac{m_\beta(\xi)}{a_\tau(\xi)}=+\infty,
$$
which implies that $\cH_{M_\beta}\nsubseteq \cH_{A_\tau}$.
\end{proof}
\end{enumerate}

We close this section with the sinc kernel (\ref{sinc}).

\begin{coro}\label{sinckernel}
There holds for all $\gamma>0$ and $d\in\bN$ that
\begin{equation}\label{estimatesincgaussian}
\lambda(\sinc,G_\gamma)=\frac{\exp(\frac{d\gamma\pi^2}4)}{(\gamma\pi)^{\frac d2}},\ \ \lambda(\sinc,A_\tau)=\frac{\sqrt{\pi}\exp(\frac{\tau\pi^2}4)}{2^{d-1}\pi^d d\sqrt{\tau}}.
\end{equation}
Consequently, $\cH_{\sinc}\subseteq\cH_K$ for $K=E_{\sigma_1},\cE_{\sigma_2}$, and $M_\beta$.
\end{coro}
\begin{proof}
Equation (\ref{estimatesincgaussian}) follows from a straightforward calculation.
\end{proof}

\section{Hilbert-Schmidt Kernels}
\setcounter{equation}{0}

By Mercer's theorem \cite{Mercer}, Hilbert-Schmidt kernels represent a large class of reproducing kernels. They were recently used to construct multiscale kernels based on wavelets \cite{Opfer}. We introduce the general form of Hilbert-Schmidt kernels.

Let $a$ be a nonnegative function on $\bN$ and set $a_n := a(n)$, $n\in\bN$. We denote by $\ell_a^2(\bN)$ the Hilbert space
of functions $c$ on $\bN$ such that
$$
\|c\|_{\ell^2_a(\bN)}:=\biggl(\sum_{n=1}^\infty a_n|c_n|^2\biggr)^{1/2} < +\infty.
$$
Its inner product is given by
$$
(c, d)_{\ell_a^2(\bN)} :=  \sum_{n=1}^\infty a_nc_n\overline{d_n}, \quad c,d\in \ell_a^2(\bN).
$$
Suppose that we have a sequence of functions $\phi_n$, $n\in\bN$, on the input space $X$, such that for each
$x\in X$ the function $\Phi(x)$ defined on $\bN$ as
\begin{equation}\label{phinx}
\Phi(x)(n) := \phi_n(x), \quad n\in\bN
\end{equation}
belongs to $\ell_a^2(\bN)$. The Hilbert-Schmidt kernel $K_a$ associate with $a$ is given as
\begin{equation}\label{kernelka}
K_a(x, y) := (\Phi(x), \Phi(y)_{\ell_a^2(\bN)} = \sum_{n=1}^\infty a_n\phi_n(x)\overline{\phi_n(y)},\quad x,y\in X.
\end{equation}
Now suppose that there exits another nonnegative function $b$ on $\bN$ such that
$\Phi(x) \in \ell_b^2(\bN)$ for all $x\in X$. Set
\begin{equation}\label{kernelkb}
K_b(x, y): = (\Phi(x), \Phi(y)_{\ell_b^2(\bN)} = \sum_{n=1}^\infty b_n\phi_n(x)\overline{\phi_n(y)},\quad x,y\in X.
\end{equation}
We shall characterize $\cH_{K_a} \subseteq \cH_{K_b}$ in terms of $a$ and $b$.

\begin{prop}\label{characterizeHilbertSchmidt}
Suppose that $b$ is nontrivial, and $\span\{\Phi(x):x\in X\}$ is dense in both $\ell_a^2(\bN)$ and $\ell_b^2(\bN)$. Then $\cH_{K_a} \subseteq \cH_{K_b}$ if and only if there is a constant $\lambda > 0$ such that $a_n\le \lambda b_n$ for all $n\in\bN$. In this case,
\begin{equation}\label{lambdaHilbertSchmidt}
\lambda(K_a,K_b)=\sup\left\{\frac{a_n}{b_n}:n\in\bN,\ b_n>0\right\}.
\end{equation}
\end{prop}
\begin{proof}
By Lemma \ref{featuremapconstruction}, the space $\cH_{K_a}$ consists of functions of the form
\begin{equation}\label{characterizeHilbertSchmidteq1}
f_c(x):=(c,\Phi(x))_{\ell^2_a(\bN)}=\sum_{n=1}^\infty c_na_n\overline{\phi_n(x)},\ \ x\in X,\ c\in\ell^2_a(\bN)
\end{equation}
with the norm
$$
\|f_c\|_{\cH_{K_a}}=\|c\|_{\ell^2_a(\bN)}.
$$
Similarly, one has the structure of the space $\cH_{K_b}$.

Suppose that there exists some constant $\lambda>0$ such that $a_n\le\lambda b_n$ for all $n\in\bN$. Let $c$ be an arbitrary but fixed element in $\ell^2_a(\bN)$ and set
$$
\tilde{c}_n:=\left\{
\begin{array}{ll}
0,&\mbox{ if }a_n=0,\\
\frac{a_nc_n}{b_n},&\mbox{ otherwise.}
\end{array}
\right.
$$
One sees that $\tilde{c}\in\ell^2_b(\bN)$ and that $(\tilde{c},\Phi(\cdot))_{\ell^2_b(\bN)}=f_c$. Thus, $f_c\in\cH_{K_b}$, implying that $\cH_{K_a}\subseteq \cH_{K_b}$. Another observation is that
$$
\|f_c\|_{\cH_{K_b}}^2=\sum_{n\in\bN,\,a_n\ne0}\left|\frac{a_nc_n}{b_n}\right|^2b_n\le \sup\{a_n/b_n:n\in\bN,\ a_n\ne0\}\|f_c\|_{\cH_{K_a}}^2.
$$
Moreover, for any $k\in\bN$ with $a_k>0$, the particular choice $c(n):=\delta_{n,k}$, $n\in\bN$, where $\delta_{n,k}$ denotes the Kronecker delta, yields that
$$
\|f_c\|_{\cH_{K_b}}^2=\frac{a_k^2}{b_k}=\frac{a_k}{b_k}\|f_c\|_{\cH_{K_a}}^2.
$$
The above two equations together imply by Proposition \ref{twoquantities} that
$$
\lambda(K_a,K_b)=\beta(K_a,K_b)^2=\sup\left\{\frac{a_n}{b_n}:n\in\bN,\ b_n>0\right\}.
$$

Conversely, suppose that $\cH_{K_a}\subseteq\cH_{K_b}$. As the embedding operator is bounded, there exists $\lambda>0$ such that
$\|f\|_{\cH_{K_b}}\le \lambda \|f\|_{\cH_{K_a}}$ for all $f\in\cH_{K_a}$.  For any $k\in\bN$ with $a_k>0$, we still choose $c(n):=\delta_{n,k}$, $n\in\bN$ to get from $f_c\in\cH_{K_b}$ that $b_k>0$ and that
$$
\|f_c\|_{\cH_{K_b}}^2=\frac{a_k^2}{b_k^2}b_k\le \lambda \|f_c\|_{\cH_{K_a}}=\lambda a_k,
$$
which implies that $a_k\le \lambda b_k$. The proof is complete.
\end{proof}

Before we give examples of inclusion relations for Hilbert-Schmidt kernels by Proposition \ref{characterizeHilbertSchmidt}, we remark that Proposition \ref{characterizeHilbertSchmidt} actually leads to a characterization of Hilbert-Schmidt kernels.

\begin{theorem}\label{HilbertSchmidt}
Let $r$ be a function on $\bN$. Suppose that $\Phi(x)\in\ell^2_{|r|}(\bN)$ for all $x\in X$ and $\span\Phi(X)$ is dense in $\ell^2_{|r|}(\bN)$. Then
\begin{equation}\label{kernelkr}
K_r(x, y) := \sum_{n=1}^\infty r_n\phi_n(x)\overline{\phi_n(y)},\quad x,y\in X
\end{equation}
defines a kernel on $X$ if and only if $r_n \geq 0$ for each $n\in \bN$.
\end{theorem}
\begin{proof}
The sufficiency is well-known. We prove the necessity by contradiction. Assume that $K_r$ given by (\ref{kernelkr})
is a kernel but $r_{j_0} < 0$ for some $j_0 \in \bN$. Then we introduce two nonnegative functions
$a$ and $b$ on $\bN$ by setting
\[a_n=
\left\{
\begin{array}{ll}
2|r_n|,& n \neq j_0,\\
-r_{j_0},& n = j_0.
 \end{array}\right.
\]
and
\[b_n=
\left\{
\begin{array}{ll}
2|r_n|+r_n,& n \neq j_0,\\
0,& n =  j_0.
 \end{array}\right.
\]
Then it is clear that $\Phi(x) \in \ell_a^2(\bN)$ and $\Phi(x) \in \ell_b^2(\bN)$ for all $x\in X$. Moreover, $\span\Phi(X)$ is dense in $\ell^2_a(\bN)$ and $\ell^2_b(\bN)$ as it is in $\ell^2_{|r|}(\bN)$. Therefore, $K_a$ and $K_b$ are Hilbert-Schmidt kernels on $X$. Note that $K_b-K_a=K_r$. By the assumption, $K_a\ll K_b$. Thus by
Proposition \ref{characterizeHilbertSchmidt}, there exists some $\lambda>0$ such that $a_n \leq \lambda b_n$ for all $n\in \bN$. Especially when
$n = j_0$, we have $-r_{j_0}\leq \lambda0=0 $, contradicting that $r_{j_0}<0$.
\end{proof}

As an application of the above theorem, we discuss an important and celebrated result which was proved before by rather sophisticated mathematical analysis \cite{SIJ2}. Suppose that the power series
$$
\sum_{n=0}^\infty a_n z^n,\ \ z\in\bC
$$
has a positive convergence radius $r$. Then by Corollary \ref{HilbertSchmidt} or \cite{SIJ2},
$$
K(x,y):=\sum_{n=0}^\infty a_n (x,y)^n, \ \ x,y\in\bR^d,\ \|x\|,\|y\|<r^{1/2}
$$
is a reproducing kernel on $\{x\in\bR^d:\|x\|<r^{1/2}\}$ if and only if $a_n\ge0$ for all $n\ge0$.

We close this section with a few examples that fall into the consideration of Proposition \ref{characterizeHilbertSchmidt}. We shall not state the results explicitly as they would just be repetition of those in Proposition \ref{characterizeHilbertSchmidt}.

\begin{itemize}
\item[--](Discrete Exponential Kernels) Let $t_n$, $n\in\bN$ be a sequence of pairwise distinct points in $\bR^d$ and let $a,b$ be two nonnegative functions in $\ell^1(\bN)$. The associated discrete exponential kernels are given by
$$
K_a(x, y): = \sum_{n=1}^\infty a_ne^{i(x - y, t_n)}, \quad K_b(x, y): = \sum_{n=1}^\infty b_ne^{i(x - y, t_n)},\ \ x,y\in\bR^d.
$$
Useful examples of discrete exponential kernels including the periodic kernels (see, for example, \cite{SS}, page 103). We present three instances below. Let $\gamma,\sigma$ be positive constants and $\alpha>d$. Define
$$
\tilde{G}_\gamma(x,y):=\sum_{n\in\bZ^d}e^{i(x-y,n)}e^{-\gamma \|n\|^2},\ \ x,y\in [0,2\pi]^d,
$$
$$
 \tilde{E}_\sigma(x,y):=\sum_{n\in\bZ^d}e^{i(x-y,n)}e^{-\sigma \|n\|},\ \ x,y\in [0,2\pi]^d,
 $$
 and
 $$
 \tilde{P}_q(x,y):=\sum_{n\in\bZ^d}e^{i(x-y,n)}\frac1{(1+\|n\|)^\alpha},\ \ x,y\in [0,2\pi]^d.
 $$
 Then by Proposition \ref{characterizeHilbertSchmidt}, we clearly have that $\cH_{\tilde{G}_\gamma}\subseteq \cH_{\tilde{E}_\sigma}\subseteq \cH_{\tilde{P}_q}$.

\item[--] (Polynomial Kernels) Let $a,b$ be two nonnegative functions on $\bN_+:=\bN\cup\{0\}$. Suppose that $\sum_{n=0}^\infty a_n z^n$ and $\sum_{n=0}^\infty b_nz^n$ both have a positive convergence radius $r_a$ and $r_b$, respectively. Then the polynomial kernels
$$
K_a(x, y): = \sum_{n = 0}^{\infty}a_n(x, y)^n,
$$
$$
K_b(x, y) := \sum_{n=0}^{\infty}b_n(x, y)^n,
$$
on the input space $\{x\in\bR^d:\|x\|<\min(\sqrt{r_a},\sqrt{r_b})\}$ satisfy the assumptions of Proposition \ref{characterizeHilbertSchmidt}.

\end{itemize}

Especially, we have the following simple observation about finite polynomial kernels.
\begin{prop}\label{finitepolynomial}
(Finite Polynomial Kernels) Let $p, q\in\bN$ and put
\begin{equation}\label{polynomialkernelp}
K_p(x, y):= (1 + (x, y))^p, \quad x, y\in \bR^d
\end{equation}
and
\begin{equation}\label{polynomialkernelq}
K_q(x, y) := (1 + (x, y))^q, \quad x, y\in \bR^d.
\end{equation}
Then $\cH_{K_p} \subseteq \cH_{K_q}$ if and only if $p\le q$. When $p\le q$, $\lambda(K_p,K_q)=1$.
\end{prop}

\section{Constructional Results}
\setcounter{equation}{0}

In this section, we discuss the preservation of the inclusion relation of RKHS under various operations with the corresponding kernels. We start with some trivial observations from Lemma \ref{Aronszajn}.

\begin{prop}\label{trivial}
Let $K_1,K_2,G_1,G_2,K,G$ be reproducing kernels on the input space $X$. Then the following results hold true:
\begin{enumerate}[i.)]
\item If $\cH_{K_1} \subseteq \cH_{G_1}$ and $\cH_{K_2} \subseteq \cH_{G_2}$ then $\cH_{K_1+K_2}\subseteq \cH_{G_1+G_2}$ and
    $$
    \lambda(K_1+K_2,G_1+G_2)\le \max(\lambda(K_1,G_1),\lambda(K_2,G_2)).
    $$
\item Especially, if $\cH_{K_1}$ and $\cH_{K_2}$ are both contained in $\cH_G$ then $\cH_{K_1+K_2}\subseteq\cH_G$ and
    $$
    \lambda(K_1+K_2,G)\le \lambda(K_1,G)+\lambda(K_2,G).
    $$
\item If $\cH_K\subseteq\cH_G$ then for all $a,b>0$, $\cH_{aK}\subseteq\cH_{bG}$ and
    $$
    \lambda(aK,bG)=\frac ab\lambda(K,G).
    $$
\end{enumerate}
\end{prop}

We next turn to the product of two kernels by first examining the more general tensor product of kernels. Let $K,G$ be two kernels on $X$. The tensor product $K\otimes G$ of $K,G$ is a new kernel on the extended input space $X\times X$ defined by
$$
(K\otimes G)(\bx,\by):=K(\bx_1,\by_1)G(\bx_2,\by_2),\ \ \bx=(\bx_1,\bx_2),\ \by=(\by_1,\by_2)\in X\times X.
$$
For further discussion, we shall make use of the Schur product theorem \cite{HJ}. For two square matrices $A,B$ of the same size, we denote by $A\circ B$ the Hadamard product of $A,B$, that is, $A\circ B$ is formed by pairwise multiplying elements from $A$ and $B$. The Schur product theorem asserts that the Hardmard product of two positive semi-definite matrices is still positive semi-definite.


\begin{prop}\label{tensorproduct}
Let $K_1,K_2,G_1,G_2$ be kernels on $X$. If $\cH_{K_1} \subseteq \cH_{G_1}$ and $\cH_{K_2} \subseteq \cH_{G_2}$ then $\cH_{K_1\otimes K_2} \subseteq \cH_{G_1\otimes G_2}$ and
$$
\lambda(K_1\otimes K_2,G_1\otimes G_2)\le\lambda(K_1,G_1)\lambda(K_2,G_2).
$$
\end{prop}
\begin{proof} For notational simplicity, put $\lambda_1:=\lambda(K_1,G_1)$ and $\lambda_2:=\lambda(K_2,G_2)$. We shall show that $K_1\otimes K_2\ll \lambda_1\lambda_2 G_1\otimes G_2$ by definition. Let $\bz:=\{\bx^j:j\in\bN_n\}$ be a finite set of pairwise distinct points in $X\times X$. Set $\bz_1:=\{\bx^j_1:j\in\bN_n\}$ and $\bz_2:=\{\bx^j_2:j\in\bN_n\}$. We observe that
$$
(G_1\otimes G_2)[\bz]=G_1[\bz_1]\circ G_2[\bz_2],\ (K_1\otimes K_2)[\bz]=K_1[\bz_1]\circ K_2[\bz_2].
$$
By Proposition \ref{twoquantities}, $K_1\ll \lambda_1 G_1$ and $K_2\ll \lambda_2 G_2$. As a result, $\lambda_1G_1[\bz_1]-K_1[\bz_1]$ and $\lambda_2G_2[\bz_2]-K_2[\bz_2]$ are both positive semi-definite. We now compute that
$$
\begin{array}{l}
 \lambda_1\lambda_2 (G_1\otimes G_2)[\bz]-(K_1\otimes K_2)[\bz]=\lambda_1\lambda_2G_1[\bz_1]\circ G_2[\bz_2]-K_1[\bz_1]\circ K_2[\bz_2]\\
\quad=( K_1[\bz_1]+(\lambda_1G_1[\bz_1]-K_1[\bz_1]))\circ ( K_2[\bz_2]+(\lambda_2G_2[\bz_2]-K_2[\bz_2]))-K_1[\bz_1]\circ K_2[\bz_2]\\
\quad=K_1[\bz_1]\circ(\lambda_2G_2[\bz_2]-K_2[\bz_2])+(\lambda_1G_1[\bz_1]-K_1[\bz_1])\circ K_2[\bz_2]\\
\quad\quad +(\lambda_1G_1[\bz_1]-K_1[\bz_1])\circ (\lambda_2G_2[\bz_2]-K_2[\bz_2]).
\end{array}
$$
By the Schur product theorem, the three matrices in the last step above are all positive semi-definite. Therefore, $K_1\otimes K_2\ll \lambda_1\lambda_2 G_1\otimes G_2$. The proof is complete.
\end{proof}

\begin{coro}\label{product}
Let $K_1,K_2,G_1,G_2$ be kernels on $X$. If $\cH_{K_1} \subseteq \cH_{G_1}$ and $\cH_{K_2} \subseteq \cH_{G_2}$ then $\cH_{K_1K_2} \subseteq \cH_{G_1G_2}$ and $\lambda(K_1 K_2,G_1G_2)\le\lambda(K_1,G_1)\lambda(K_2,G_2)$.
\end{coro}
\begin{proof}
The result follows from Proposition \ref{tensorproduct} and the observation that $K_1K_2$ and $G_1G_2$ can be viewed as the restriction of $K_1\otimes K_2$ and $G_1\otimes G_2$ on the diagonal of $X\times X$, respectively.
\end{proof}

We next discuss limits of reproducing kernels. It is obvious by definition that the limit of a sequence of kernels remains a kernel \cite{Aronszajn}.

\begin{prop}\label{limit}
Let $\{K_j:j\in\bN\}$ and $\{G_j:j\in\bN\}$ be two sequences of kernels on $X$ that converge pointwise to kernels $K$ and $G$, respectively. If $\cH_{K_j}\subseteq\cH_{G_j}$ for all $j\in\bN$ and
\begin{equation}\label{limitcondition}
 \sup\{\lambda(K_j,G_j):j\in\bN\}<+\infty
\end{equation}
then $\cH_{K} \subseteq \cH_{G}$ and $\lambda(K,G)\le \sup\{\lambda(K_j,G_j):j\in\bN\}$.
\end{prop}
\begin{proof}
Suppose that $\cH_{K_j}\subseteq\cH_{G_j}$ for all $j\in\bN$ and $\lambda:=\sup\{\lambda(K_j,G_j):j\in\bN\}<+\infty$. Let $\bx$ be a finite set of sampling points in $X$ and $y\in\bC^n$ be fixed. Then as $K_j\ll \lambda G_j$, we have for all $j\in\bN$ that
$$
y^*(\lambda G_j[\bx] - K_j[\bx])y \geq 0.
$$
Taking the limit as $j\to\infty$, we get that
$$
y^*(\lambda G[\bx] - K[\bx])y \geq 0.
$$
The proof is hence complete.
\end{proof}

We remark that condition (\ref{limitcondition}) may not be removed in the last proposition. For a simple contradictory example, we let $G$ be an arbitrary nontrivial kernel on $X$ and set $K_j:=\frac 1j G$ and $G_j:=G$ for all $j\in\bN$. It is cleat that $\cH_{K_j}=\cH_{G_j}=\cH_G$ for each $j\in\bN$. But the limit of $K_j$ is the trivial kernel. The inclusion relation is hence not kept in the limit kernels. The reason is that $\lambda(K_j,G_j)=j$ is unbounded.

With the help of Propositions \ref{trivial}, \ref{limit} and Corollary \ref{product}, we are ready to give a main result of this section. We shall use a fact proved in \cite{FMP} that if $K$ is a kernel and $\phi$ is analytic with nonnegative Taylor coefficients at the origin then $\phi(K)$ remains a kernel.

\begin{theorem}\label{exponentialcomposition}
Let $K$ and $G$ be two kernels on $X$ with $\cH_K\subseteq \cH_G$. Then $\cH_{e^K}\subseteq\cH_{e^{\lambda(K,G)G}}$. In particular, if $\lambda(K,G)\le 1$ then $\cH_{e^K}\subseteq\cH_{e^G}$.
\end{theorem}
\begin{proof}
We may assume that $\lambda(K,G)\le 1$. Let $K_n: = \sum_{j = 0}^n\frac{K^j}{j!}$ and $G_n: = \sum_{j = 0}^n\frac{G^j}{j!}$ for each $n\in\bN$. Then $K_n,G_n$ converge pointwise to $e^K$ and $e^G$, respectively. It also follows from
Proposition \ref{trivial} and Corollary \ref{product} that
$$
K_n \ll \left(\max_{0\leq j\leq n}\lambda(K,G)^{j}\right)G_n.
$$
It is clear that $\max_{0\leq j\leq n}\lambda(K,G)^{j}$, $n\in\bN$ are bounded by $1$. The result now follows immediately from Proposition \ref{limit}.
\end{proof}

The arguments used in the above proof in fact are able to prove a more general result, which we present below.

\begin{prop}
Let $K$ and $G$ be two kernels on $X$ with $\cH_K\subseteq \cH_G$. Suppose that $\phi$ is an analytic function with nonnegative Taylor coefficients $a_j$, $j\ge 0$ at the origin. Then $\cH_{\phi(K)}\subseteq \cH_{\phi(\lambda(K,G)G)}$. If, in addition, $\lambda(K,G)\le 1$, then $\cH_{\phi(K)}\subseteq \cH_{\phi(G)}$.
\end{prop}

\section{Equivalent Norm Inclusion}
\setcounter{equation}{0}

In this section, we investigate a special inclusion relation where an equivalence on the norms on the smaller space is imposed. Specifically, for two kernels $K,G$ on $X$, we denote by $\cH_K\lesssim \cH_G$ if $\cH_K\subseteq\cH_G$ and there exists positive constants $\alpha,\beta$ such that
\begin{equation}\label{equivalentnormdef}
\alpha\|f\|_{\cH_K}\le \|f\|_{\cH_G}\le\beta \|f\|_{\cH_K}\mbox{ for all }f\in\cH_K.
\end{equation}
For an existing kernel $K$, we call a kernel $G$ a {\it weak refinement} of $K$ if $\cH_K\lesssim \cH_G$. This is a relaxation of the refinement kernel defined in \cite{XZ2} and is expected to accommodate more examples of reproducing kernels.

We start our investigation with a characterization of the equivalent norm inclusion relation. The following result from \cite{Aronszajn} is needed.

\begin{lemma}\label{kernelsum}
Let $K$ and $G$ be kernels on $X$. Then there holds for all $f\in \cH_{K + G}$ that
$$
\|f\|_{\cH_{K+G}}^2 = \min \{\|f_1\|_{\cH_K}^2 + \|f_2\|_{\cH_G}^2:\  f=f_1+f_2,\  f_1 \in \cH_K,\ f_2 \in{\cH_G} \}.
$$
\end{lemma}

\begin{theorem}\label{equivalentnorm}
Let $K$ and $G$ be kernels on $X$ with $\cH_K \subseteq \cH_G$. Then $\cH_K \lesssim \cH_G$ if and only if there exists some constant $\delta>0$ such that
\begin{equation}\label{equivalentnormcond}
\|e\|_{\cH_{\lambda(K,G)G - K}} \geq \delta \|e\|_{\cH_K}\mbox{ for each }e\in\cH_K \cap \cH_{\lambda(K,G)G - K}.
\end{equation}
\end{theorem}
\begin{proof}
For notational simplicity, put $L := \lambda(K,G)G - K$. By Proposition \ref{twoquantities}, $L$ is a kernel on $X$. Suppose that condition (\ref{equivalentnormcond}) is satisfied. Note that for each $f\in\cH_K\subseteq\cH_G$ with the decomposition $f=f_1+f_2$ where $f_1\in\cH_K$, $f_2\in\cH_{L}$, we have $f_2\in\cH_{K}\cap\cH_L$. This together with $\cH_{K}\subseteq\cH_G=\cH_{\lambda(K,G)G}$ implies by Lemma \ref{kernelsum} that for all $f\in\cH_K$
\begin{eqnarray*}
\|f\|^2_{\cH_{\lambda(K,G)G}} & = & \min_{f = f_1 + f_2} \{\|f_1\|_{\cH_K}^2 + \|f_2\|_{\cH_L}^2: \; f_1 \in \cH_K, \;f_2 \in \cH_L\} \\
 & \geq & \min_{f = f_1 + f_2} \{\|f_1\|_{\cH_K}^2 + \delta^2\|f_2\|_{\cH_K}^2: \; f_1 \in \cH_K, \;f_2 \in{\cH_L} \}\\
 & \geq & \min_{f = f_1 + f_2}\min\{1, \delta^2\}\{\|f_1\|_{\cH_K}^2 + \|f_2\|_{\cH_K}^2: \; f_1 \in \cH_K, \;f_2 \in{\cH_L} \}\\
 & \geq & \frac{1}{2}\min\{1, \delta^2\}\|f\|_{\cH_K}^2.
\end{eqnarray*}
Recall that for all $f\in\cH_G$,
$$
\|f\|_{\cH_G}=\sqrt{\lambda(K,G)}\|f\|_{\lambda(K,G)G}.
$$
By the above two equations and Proposition \ref{twoquantities}, we have for all $f\in\cH_K$ that
$$
\frac1{\sqrt{2}}\min\{1, \delta\}\sqrt{\lambda(K,G)}\|f\|_{\cH_K}\le\|f\|_{\cH_G}\le \sqrt{\lambda(K,G)}\cH_K
$$
in other words, $\cH_K\lesssim \cH_G$.

Conversely, suppose that $\cH_K\lesssim\cH_G$ but (\ref{equivalentnormcond}) does not hold for any $\delta>0$. Then for each $n\in\bN$, there exists
$g_n \in \cH_K \cap \cH_{L}$ such that
\begin{equation}\label{equivalentnormeq2}
\|g_n\|_{\cH_{L}} \leq \frac{1}{n} \|g_n\|_{\cH_K}.
\end{equation}
Since $L \ll \lambda(K,G)G$, it follows from Lemmas \ref{Aronszajn} and \ref{kernelsum} that $\cH_L\subseteq\cH_{\lambda(K,G)G}$ and
\begin{equation}\label{equivalentnormeq3}
\|g_n\|_{\cH_{G}}=\sqrt{\lambda(K,G)}\|g_n\|_{\cH_{\lambda(K,G)G}} \leq \sqrt{\lambda(K,G)}\|g_n\|_{\cH_L}\mbox{ for all }n\in\bN.
\end{equation}
Equations (\ref{equivalentnormeq2}) and (\ref{equivalentnormeq3}) imply that
$$
\|g_n\|_{\cH_G}\le \frac{\sqrt{\lambda(K,G)}}{n}\|g_n\|_{\cH_K}\mbox{ for all }n\in\bN,
$$
contradicting (\ref{equivalentnormdef}). The proof is complete.
\end{proof}

As an application of Theorem \ref{equivalentnorm}, we have the following example.

\begin{prop}
Consider the two finite polynomial kernels $K_p,K_q$ defined by (\ref{polynomialkernelp}) and (\ref{polynomialkernelq}). Then $\cH_{K_p}\lesssim\cH_{K_q}$ if and only if $p\le q$.
\end{prop}
\begin{proof}
By Proposition \ref{finitepolynomial}, $\cH_{K_p}\subseteq\cH_{K_q}$ if and only if $p\le q$. Thus, if $\cH_{K_p}\lesssim\cH_{K_q}$ then $p\le q$. Suppose that $p\le q$. We introduce another kernel $K$ on $\bR^d$ by setting
$$
K(x,y):=\sum_{j=0}^p {{q}\choose{j}} (x, y)^j,\ \ x,y\in\bR^d.
$$
Then by Proposition \ref{characterizeHilbertSchmidt}, $\cH_K\subseteq\cH_{K_q}$ and $\cH_{K}=\cH_{K_p}$. It is clear that $\cH_K\cap{\cH_{K_q-K}}=\{0\}$. By Theorem \ref{equivalentnorm}, $\cH_K\lesssim\cH_{K_q}$. As $\cH_{K}=\cH_{K_p}$, we have $\cH_{K_p}\lesssim\cH_{K_q}$. The proof is complete.
\end{proof}

Before moving on, we make a simple observation that if two kernels $K,G$ on $X$ satisfy $\cH_K\lesssim\cH_G$ and $\cH_K\ne\cH_G$ then $\cH_K$ can not be dense in $\cH_G$. For instances, given two Gaussian kernels $G_{\gamma_1}$, $G_{\gamma_2}$ with $\gamma_1<\gamma_2$. As $\cH_{G_{\gamma_2}}\subseteq \cH_{G_{\gamma_1}}$ and $\cH_{G_{\gamma_2}}$ is dense in but not equal to $\cH_{G_{\gamma_1}}$, $G_{\gamma_1}$ is not a weak refinement of $G_{\gamma_2}$.

The main purpose of this section is to present two characterizations of the equivalent norm inclusion that are widely applicable to translation invariant kernels and Hilbert-Schmidt kernels. As the study would be similar to that in \cite{XZ2}, we shall omit the proof and examples.

Let $\mu, \nu$ be two finite positive Borel measures on a topological space $Y$. Set
$$
\omega := \frac{\mu + \nu}{2} + \frac{|\mu - \nu|}{2},
$$
where $|\mu - \nu|$ denotes the total variation measure of $\mu - \nu$. Then
$\mu$ and $\nu$ are absolutely continuous with respect to $\omega$. Given a function $\phi : X \times Y \rightarrow \bC$ such that
$\phi(x, \cdot) \in L^2_\omega(Y)$ for all $x\in X$ and
\begin{equation}\label{densenessphi}
\overline{\span }\{\phi(x, \cdot) : x\in X\} = L^2_\omega(Y),
\end{equation}
we introduce two kernels $K_{\mu}$, $K_{\nu}$ on $X$ by setting
\begin{equation}\label{kernelkmuknu}
K_{\mu}(x, y) := (\phi(x, \cdot), \phi(y, \cdot))_{L^2_\mu(Y)}, \quad
K_{\nu}(x, y) := (\phi(x, \cdot), \phi(y, \cdot))_{L^2_\nu(Y)},\ \ x,y\in X.
\end{equation}
Our task is to characterize the equivalent inclusion relation $\cH_K \lesssim \cH_G$ in terms of the measures $\mu$ and $\nu$. To this end,
we write $\mu \lesssim \nu$ if $\mu \ll \nu$ and there exist positive constants $\alpha,\beta$ such that $\alpha \leq d\mu / d\nu \leq \beta$
almost everywhere on $\{t\in Y:\ \frac{d\mu}{d\nu}(t)>0\}$ with respect to $\nu$.

The following characterization theorem can be proved by arguments similar to those in \cite{XZ2}.

\begin{theorem}\label{equivalentnormcharacterizeXZ}
Suppose that $\phi : X \times Y \rightarrow \bC$ satisfies (\ref{densenessphi}) and $K_{\mu}$,
$K_{\nu}$ are defined by (\ref{kernelkmuknu}). Then $\cH_{\mu} \lesssim \cH_{\nu}$ if and only
if $\mu \lesssim \nu$ .
\end{theorem}

The above theorem has a particular application to Hilbert-Schmidt kernels. For two nonnegative functions $a,b$ on $\bN$, we denote by $a\lesssim b$ if $\supp a \subseteq \supp b$ and
there exist two positive constants $\alpha$ and $\beta$ such that
$\alpha a_n \leq b_n \leq \beta a_n$ for each $n \in \supp a$. Here $\supp a:=\{n\in \bN:a_n\ne0\}$. Recall the definition of Hilbert-Schmidt kernels (\ref{kernelka}) and (\ref{kernelkb}) through a sequence of functions (\ref{phinx}).

\begin{prop}\label{equivalentnormcharacterizeXZHilbertSchmidt}
Suppose that $\span\{\Phi(x):x\in X\}$ is dense in both $\ell_a^2(\bN)$ and $\ell_b^2(\bN)$. Then $\cH_{K_a} \lesssim \cH_{K_b}$ if and only if $a\lesssim b $.
\end{prop}

We want to reemphasize that our results, though similar to those in \cite{XZ2}
for refinement of reproducing kernels, much increase the chance of refining an existing kernel. Taking polynomial kernels as an instance, for two such kernels
$$
K(x,y):=\sum_{j=0}^N a_j (x,y)^j,\ \ G(x,y):=\sum_{k=0}^M b_k(x,y)^k,\ \ x,y\in\bR^d,
$$
where $a_j,b_k$ are positive constants. By Proposition \ref{equivalentnormcharacterizeXZHilbertSchmidt}, $\cH_{K_a} \lesssim \cH_{K_b}$
if $N\le M$. However, asking $K_b$ to be a refinement kernel of $K_a$ would impose a strong additional requirement that $a_j=b_j$ for all $j\in\bN_N$. A more concrete example is the kernels $K_p,K_q$ appeared in (\ref{polynomialkernelp}) and (\ref{polynomialkernelq}). By our discussion, if $p<q$ then $K_q$ is a weak refinement but not a refinement of $K_q$.

\bibliographystyle{amsplain}

\begin{thebibliography}{10}

\bibitem{Aronszajn} N. Aronszajn, Theory of reproducting kernels, {\it Trans. Amer. Math. Soc.} \textbf{68} (1950), 337--404.

\bibitem{Berlinet} A. Berlinet and C. Thomas-Agnan, {\it Reproducing Kernel Hilbert Spaces in Probability and Statistics}, Kluwer Academic Publishers, Boston, MA, 2004.

\bibitem{Bochner} S. Bochner, {\it Lectures on Fourier Integrals with an author's supplement on monotonic functions,
Stieltjes integrals, and harmonic analysis}, Annals of Mathematics
Studies \textbf{42}, Princeton University Press, New Jersey, 1959.

\bibitem{CuckerSmale} F. Cucker and S. Smale, On the mathematical foundations of learning, {\it Bull. Amer. Math. Soc.}
\textbf{39} (2002), 1--49.

\bibitem{Cui2007} M. Cui and F. Geng, Solving singular two-point boundary value problem in reproducing kernel space, {\it J. Comput. Appl. Math.} \textbf{205} (2007), 6--15.

\bibitem{DR} M. F. Driscoll, The reproducing kernel Hilbert space structure of the smaple paths of a Gaussian process,
 {\it Z. Wahrsch. Verw. Geb} \textbf{26} (1973), 309--316.

\bibitem{Evgeniou2000} T. Evgeniou, M. Pontil and T. Poggio, Regularization networks and support vector machines, {\it Adv. Comput. Math.} \textbf{13} (2000), 1--50.

\bibitem{FMP} C. H. FitzGerald, C. A. Micchelli and A. Pinkus, Functions that preserve families of positive semidefinite matrices, {\it Linear Algegra Appl.} \textbf{221}, (1995), 83--102 .

\bibitem{Fukumizu2004} K. Fukumizu, F. R. Bach and M. I. Jordan, Dimensionality reduction for supervised learning with reproducing kernel Hilbert spaces, {\it J. Mach. Learn. Res.} \textbf{5} (2004), 73--99.


\bibitem{Franke1998} C. Franke and R. Schaback, Solving partial differential equations by collocation using radial basis functions, {\it Appl. Math. Comput.} \textbf{93} (1998), 73--82.


\bibitem{HJ} R. A. Horn and C. R. Johnson, {\it Topics in Matrix Analysis}, Cambridge University Press, Cambridge, 1991.


\bibitem{Mercer} J. Mercer, Functions of positive and negative type and
their connection with the theorey of integral equations, {\it
Philos. Trans. R. Soc. Lond. Ser. A Math. Phys. Eng. Sci.}
\textbf{209} (1909), 415--446.

\bibitem{NW} M. Z. Nashed and G. G. Walter, General
sampling theorems for functions in reproducing kernel Hilbert
spaces, {\it Math. Control Signals Systems} \textbf{4} (1991),
363--390.

\bibitem{Opfer} R. Opfer, Multiscale kernels, {\it Adv. Comput. Math.} \textbf{25} (2006), 357--380.


\bibitem{SIJ} I. J. Schoenberg, Metric spaces and completely monotone functions, {\it Ann. of Math.(2)} \textbf{39}, (1938), 811--841 .

\bibitem{SIJ2} I. J. Schoenberg, Positive definite functions on spheres, {\it Duke. Math. J.} \textbf{9} (1942), 96--108.

\bibitem{SS} B. Sch\"{o}lkopf and A. J. Smola, {\it Learning with Kernels: Support Vector Machines, Regularization,
Optimization, and Beyond}, MIT Press, Cambridge, 2002.


\bibitem{SC} J. Shawe-Taylor, N. Cristinini, {\it Kenel Methods for Pattern Analysis}, Cambridge University Press, Cambridge, 2004.

\bibitem{Sriperumbudur2010} B. K. Sriperumbudur, A. Gretton, K. Fukumizu, B. Sch\"{o}lkopf and G. R.G. Lanckriet, Hilbert space embeddings and metrics on probability measures, {\it J. Mach. Learn. Res.} \textbf{11} (2010), 1517--1561.

\bibitem{Stein} E. M. Stein, {\it Singular Integrals And Differentiability Properties of Functions}, Princeton University Press, Princeton, 1971.

\bibitem{Vapnik1998} V. N. Vapnik, {\it Statistical Learning Theory}, Wiley, New York, 1998.


\bibitem{Wendland} H. Wendland, {\it Scattered Data Approximation}, Cambridge University Press, Cambridge, 2005.

\bibitem{XZ} Y. Xu and H. Zhang, Refinable kernels, {\it J. Mach. Learn.
Res.} \textbf{8} (2007), 2083--2120.

\bibitem{XZ2} Y. Xu and H. Zhang, Refinement of reproducing kernels,
{\it J. Mach. Learn. Res.} \textbf{10} (2009), 107--140.

\bibitem{Ylvisaker} N. D. Ylvisaker, On linear estimation for regression problems on time series, {\it Ann. Math. Statist.} \textbf{33} (1962), 1077--1084.

\bibitem{Zhanga} H. Zhang and J. Zhang, Frames, Riesz bases, and sampling expansions in Banach spaces via semi-inner products, {\it Appl. Comput. Harmon. Anal.} \textbf{31} (2011), 1--25.
\end{thebibliography}

\end{document}